\documentclass[]{amsart}
\usepackage[utf8]{inputenc}
\usepackage[mathscr]{eucal} 
\usepackage{stmaryrd} 
\usepackage{amsmath,amsthm,amsfonts,amssymb,amscd} 
\usepackage[dvipsnames]{xcolor} 
\usepackage{xspace} 
\usepackage{fancyhdr} 
\usepackage{graphicx} 
\usepackage{listings} 
\usepackage[]{hyperref} 
\usepackage{enumitem} 
\usepackage{relsize} 
\usepackage{tikz-cd} 
\usepackage{mdwlist} 
\usepackage{multicol} 
\usepackage{float} 
\usepackage{adjustbox} 
\usepackage{tikz} 
\usepackage[noabbrev,nameinlink]{cleveref} 
\usepackage{setspace} 
\usepackage{bbm} 
\usepackage{mathtools} 
\usepackage{epigraph} 
\usetikzlibrary{matrix, calc, arrows}

\hypersetup{
    colorlinks=true, 
    linktoc=all,     
    linkcolor=Brown,citecolor=Brown,urlcolor=MidnightBlue,  
}
\setlist[enumerate]{label=(\alph*)} 
\usetikzlibrary{graphs,decorations.pathmorphing,decorations.markings}
\tikzcdset{scale cd/.style={every label/.append style={scale=#1},
        cells={nodes={scale=#1}}}}

\newcommand{\Hom}{\operatorname{Hom}}
\newcommand{\Aut}{\operatorname{Aut}}

\newcommand{\tr}{\operatorname{tr}}
\newcommand{\rk}{\operatorname{rk}}
\newcommand{\inv}{^{-1}}

\newcommand{\End}{\operatorname{End}}

\newcommand{\Res}{\operatorname{Res}}
\newcommand{\Ind}{\operatorname{Ind}}
\newcommand{\Inf}{\operatorname{Inf}}

\newcommand{\id}{\operatorname{id}}

\newcommand{\im}{\operatorname{im}}

\newcommand{\Syl}{\operatorname{Syl}}

\newcommand{\Spc}{\operatorname{Spc}}
\newcommand{\Pic}{\operatorname{Pic}}
\newcommand{\Irr}{\operatorname{Irr}}

\newcommand{\open}{\operatorname{open}}
\newcommand{\cone}{\operatorname{cone}}

\newcommand{\on}[1]{\operatorname{#1}}

\newcommand{\calE}{\mathcal{E}}
\newcommand{\calF}{\mathcal{F}}

\newcommand{\calM}{\mathcal{M}}
\newcommand{\calN}{\mathcal{N}}
\newcommand{\calO}{\mathcal{O}}

\newcommand{\catB}{\mathscr{B}}

\newcommand{\catD}{\mathscr{D}}

\newcommand{\catF}{\mathscr{F}}

\newcommand{\catI}{\mathscr{I}}

\newcommand{\catK}{\mathscr{K}}
\newcommand{\catL}{\mathscr{L}}
\newcommand{\catM}{\mathscr{M}}

\newcommand{\catP}{\mathscr{P}}
\newcommand{\catQ}{\mathscr{Q}}
\newcommand{\catR}{\mathscr{R}}

\newcommand{\catT}{\mathscr{T}}

\newcommand{\fraka}{\mathfrak{a}}

\newcommand{\frakm}{\mathfrak{m}}

\newcommand{\frakp}{\mathfrak{p}}
\newcommand{\frakq}{\mathfrak{q}}

\newcommand{\bbC}{\mathbb{C}}

\newcommand{\bbF}{\mathbb{F}}

\newcommand{\bbN}{\mathbb{N}}

\newcommand{\bbR}{\mathbb{R}}

\newcommand{\bbZ}{\mathbb{Z}}
\newcommand{\bbone}{\mathbbm{1}}

\newcommand{\Weyl}[2]{{#1}/\!\!/{#2}}

\newcommand{\comp}{\on{comp}}

\makeatletter
\newcommand*{\doublerightarrow}[2]{\mathrel{
        \settowidth{\@tempdima}{$\scriptstyle#1$}
        \settowidth{\@tempdimb}{$\scriptstyle#2$}
        \ifdim\@tempdimb>\@tempdima \@tempdima=\@tempdimb\fi
        \mathop{\vcenter{
                \offinterlineskip\ialign{\hbox to\dimexpr\@tempdima+1em{##}\cr
                    \rightarrowfill\cr\noalign{\kern.5ex}
                    \rightarrowfill\cr}}}\limits^{\!#1}_{\!#2}}}
\newcommand*{\triplerightarrow}[1]{\mathrel{
        \settowidth{\@tempdima}{$\scriptstyle#1$}
        \mathop{\vcenter{
                \offinterlineskip\ialign{\hbox to\dimexpr\@tempdima+1em{##}\cr
                    \rightarrowfill\cr\noalign{\kern.5ex}
                    \rightarrowfill\cr\noalign{\kern.5ex}
                    \rightarrowfill\cr}}}\limits^{\!#1}}}
\makeatother

\newcommand{\supp}{\on{supp}}
\newcommand{\sbull}{{\scriptscriptstyle\bullet}}
\newcommand{\tH}{\on{H}^{\sbull\sbull}}

\newcommand{\Spec}{\on{Spec}}

\newcommand{\Proj}{\on{Proj}}

\newcommand{\perm}{\mathsf{perm}}
\newcommand{\Sub}{\on{Sub}}
\newcommand{\reg}{{}_{\on{reg}}}

\newtheorem{theorem}{Theorem}[section]
\newtheorem{lemma}[theorem]{Lemma}
\newtheorem{proposition}[theorem]{Proposition}
\newtheorem{corollary}[theorem]{Corollary}

\newtheorem*{theorem*}{Theorem}

\newtheorem*{question*}{Question}

\theoremstyle{definition}
\newtheorem{construction}[theorem]{Construction}

\theoremstyle{remark}
\newtheorem{remark}[theorem]{Remark}
\newtheorem{example}[theorem]{Example}
\newtheorem{observation}[theorem]{Observation}
\newtheorem{recollection}[theorem]{Recollection}
\newtheorem{notation}[theorem]{Notation}

\theoremstyle{definition}
\newtheorem{definition}[theorem]{Definition}

\makeatletter
\newcommand{\xhookdoubleheadrightarrow}[2][]{%
    \lhook\joinrel
    \ext@arrow 0359\rightarrowfill@ {#1}{#2}%
    \mathrel{\mspace{-15mu}}\rightarrow
}
\makeatother

\newtheorem{innercustomthm}{Theorem}
\newenvironment{customthm}[1]
{\renewcommand\theinnercustomthm{#1}\innercustomthm}
{\endinnercustomthm}

\AddToHook{env/corollary/begin}{\crefalias{theorem}{corollary}}
\AddToHook{env/proposition/begin}{\crefalias{theorem}{proposition}}
\AddToHook{env/lemma/begin}{\crefalias{theorem}{lemma}}
\AddToHook{env/definition/begin}{\crefalias{theorem}{definition}}
\AddToHook{env/example/begin}{\crefalias{theorem}{example}}
\AddToHook{env/remark/begin}{\crefalias{theorem}{remark}}
\AddToHook{env/observation/begin}{\crefalias{theorem}{observation}}
\AddToHook{env/construction/begin}{\crefalias{theorem}{construction}}
\AddToHook{env/recollection/begin}{\crefalias{theorem}{recollection}}
\AddToHook{env/notation/begin}{\crefalias{theorem}{notation}}

\begin{document}
    \title[Integral permutation twisted cohomology]{Integral permutation twisted cohomology: one auspicious twist is all you need}
    \author{Juan Omar G\'omez}
    \author{Sam K. Miller}

    \makeatletter
    \patchcmd{\@setaddresses}{\indent}{\noindent}{}{}
    \patchcmd{\@setaddresses}{\indent}{\noindent}{}{}
    \patchcmd{\@setaddresses}{\indent}{\noindent}{}{}
    \patchcmd{\@setaddresses}{\indent}{\noindent}{}{}
    \makeatother

    \address{Juan Omar G\'omez, Fakultat f\"ur Mathematik, Universit\"at Bielefeld, D-33501 Bielefeld, Germany}
    \email{jgomez@math.uni-bielefeld.de}
    \urladdr{https://sites.google.com/cimat.mx/juanomargomez/home}

    \address{Sam K. Miller, Department of Mathematics, University of Georgia, Athens GA 30602, United States of America}
    \email{sam.miller@uga.edu}
    \urladdr{https://www.samkmiller.com/}
    \begin{abstract}
        We introduce a reduced form of permutation twisted cohomology for the derived category of permutation modules over any commutative ring. Unlike the twisted cohomology rings previously constructed by Balmer--Gallauer and Miller, this reduced twisted cohomology ring is only bigraded, twisted by a single, well-chosen endotrivial complex corresponding to the representation sphere arising from the (doubled) regular real representation of a group. We prove that over any commutative ring, the Balmer spectrum of the derived category of permutation modules is homeomorphic to the projectivization of the spectrum of the reduced twisted cohomology ring. Moreover, we show that if the base ring is Noetherian, this ring is finite over the finitely generated subalgebra generated by the forerunners, certain $p$-local isomorphisms, of the twist. We apply these results to further analyze the geometry of permutation modules, determine the local Picard group of the geometric open of the Balmer spectrum, and resolve other questions posed by Miller. We include an appendix giving a general criterion for finite generation of twisted cohomology rings in tensor-triangular geometry.

    \end{abstract}

    \maketitle

    \epigraph{\textit{Simplicity does not need to be simple. Instead it should be forged of
        complexity that has been compressed and synthesized.}}{Alfred Jarry, \textit{Les Minutes de Sable M\'emorial} \cite{Jar94}}

    \section*{Introduction}

    Let $G$ be a finite group and $k$ a field of prime characteristic $p$. The so-called \emph{derived category of permutation modules} \[\catD(G;k) = \mathsf{DPerm}(G;k),\] and its compact part, the bounded homotopy category of $p$-permutation modules \[\catD(G;k)^c = \catK(G;k) = \on{K}_b(\mathsf{perm}(G;k)^\natural),\] have concrete realizations in the languages of representation theory, algebraic topology, homotopy theory, and motivic geometry, see \cite{BG25b, F25}. Despite the down-to-earth nature of the compact part, the tensor-triangular geometry of $\catK(G;k)$, as deduced by Balmer--Gallauer, is considerably more involved than its conventional derived counterpart $\on{D}_b(kG) \coloneqq  \on{D}_b(\mathsf{mod}(kG))$. Indeed, every $kG$-module, of which there are usually innumerably many, admits a finite resolution by $p$-permutation modules \cite{BG23a}. So while the monoidal inclusion \[\mathsf{perm}(G;k)^\natural \hookrightarrow \mathsf{mod}(kG)\] embeds a category with finitely many indecomposables into one with wildly many, the induced tt-functor \[\on{K}_b(\mathsf{perm}(G;k)^\natural) \twoheadrightarrow \on{D}_b(\mathsf{mod}(kG))\]
    is, rather surprisingly, a Verdier localization. Geometrically, the derived category openly immerses into the derived category of permutation modules. In fact, as a set, $\Spc(\catK(G;k))$ is built from $p$-local group cohomological support varieties \[\Spc(\catK(G;k)) = \bigsqcup_{H \in \Sub_p(G)/G} \Spc(\on{D}_b(k(\Weyl{G}{H}))) \cong \bigsqcup_{H \in \Sub_p(G)/G} \Spec^h(\on{H}^\sbull(\Weyl{G}{H};k)),\] where $\Weyl{G}{H} = N_G(H)/H$ is the Weyl group of $H$. That is, the primes are parametrized by pairs $(H,\frakp)$ with $H$ a representative of a conjugacy class of $p$-subgroups of $G$, and $\frakp$ a homogeneous prime in the cohomology of $\Weyl{G}{H}$ \cite{BG25}.

    \subsection*{Topology and permutation twisted cohomology theories} The question of determining how these $p$-local cohomological pieces are glued together is subtler. We review prior methods its deduction. Chronologically, Balmer--Gallauer \cite{BG25} first resolved this question in two steps. First, they established a colimit theorem a la Quillen, expressing $\Spc(\catK(G;k))$ as a colimit of subquotients $\Spc(\catK(H/K;k))$ of pairs of subgroups $H,K$ with $H/K$ an elementary abelian $p$-group. From there, the duo inspected the geometry of $\catK(E;k)$ for $E$ elementary abelian via a \emph{permutation twisted cohomology ring}, adapting a construction from algebraic geometry. This is a multigraded ring $\tH_{\calE_p}(E;k)$ consisting of morphisms from the tensor unit to particular `tensor-triangular line bundles,' invertible objects which become locally trivial with respect to an open cover of the Balmer spectrum.

    From here, one deduces the Balmer spectrum locally in each open, as the twisted cohomology ring $\tH_{\calE_p}(E;k)$ identifies with the usual cohomology ring, the graded endomorphism ring of the tensor unit, as well as the local Balmer spectrum. In particular, $\Spc(\catK(E;k))$ is a Dirac scheme in the sense of \cite{HP25}. From this, one obtains an open immersion $\Spc(\catK(E;k)) \hookrightarrow \Spec^h(\tH_{\calE_p}(E;k))$. Some looming questions remain:

    \begin{question*}
        Do the geometric properties of $\Spc(\catK(E;k))$ for $E$ an elementary abelian $p$-group extend beyond this case? How big is the closed complement of the image of $\Spc(\catK(E;k))$?
    \end{question*}

    Miller provided a path forward with the deduction of the Picard group of the derived category of permutation modules $\Pic(\catK(G;k))$ \cite{Mil25b}. In particular, if $S$ is a $p$-group, the invertible objects can be viewed as `virtual' representation spheres, and moreover, all invertibles, which Miller termed \emph{endotrivial complexes}, are tensor-triangular line bundles. Given this fact, he defined the \emph{remixed} twisted cohomology ring \cite{Mil25c}, which we denote $\tH_{\Pic^+}(S;k)$ for $p$-groups $S$, by twisting by the endotrivials associated to the irreducible real representation (spheres) of $S$, which describe a canonical basis of the Picard group. In fact, this ring is nothing more than the representation-graded Bredon homology ring of $S$ (although passing from $\on{RO}(S)$ to $\Pic(\catK(S;k))$ has nontrivial kernel in general, so one needs to be precise about what `representation-graded' means).

    If one knows that $\tH_{\Pic^+}(S;k)$ is Noetherian, then one recovers similar results of Balmer--Gallauer, but Miller was unable to deduce Noetherianity in full generality, leaving it as a conjecture. Moreover, it remained unclear if the open cover completely separates the closed points of $\Spc(\catK(S;k))$ which are enumerated by conjugacy classes of $p$-subgroups; one should hope this occurs, as the opens should each correspond to a `$p$-local layer' of sorts.

    \begin{question*}
        Is the remixed twisted cohomology ring Noetherian? Does the associated open cover completely separate the closed points of the spectrum? Do the geometric phenomena here extend beyond the $p$-group case?
    \end{question*}

    Finally, G\'omez and Dubey--G\'omez investigated the geometry of permutation modules over a commutative Noetherian ring $k$ \cite{DG25, Gom25}. Via the theory of fiber functors, $\Spc(\catK(G;k))$ is built as a set of the fibers of $k$, that is,
    \[\Spc(\catK(G;k)) = \bigsqcup_{\fraka \in \Spec(k)} \Spc(\catK(G;k(\fraka))),\]
    where $k(\fraka)$ is the fiber of $\fraka$, i.e., the field of fractions of $k/\fraka$. As expected, describing how these pieces glue together is more subtle. The method proposed by the authors consists in several reductions. By standard methods, one can reduce the problem to $p$-groups. In this case, the fiber above the prime $(p)$ under the usual comparison map gives us a distinguished closed subset called the \textit{modular locus}. Its complement is easier to understand since it's homeomorphic to $\Spc(k[1/p])$. In fact, the closure of this open subset is entirely contained in the cohomological open
    \[
    \Spec^h(\on{H}^\sbull(G;k))\cong V_{G;k}\subseteq \Spc(\catK(G;k))
    \]
    hence one is left to understand the modular locus. Under mild assumptions on $k$, the authors show that a further reduction to commutative Noetherian rings of characteristic $p$ is possible. For such rings, a similar strategy as in the one in \cite{BG25} is used. In particular, one focuses on elementary abelian $p$-groups, and for those a permutation twisted cohomology theory $\tH_{\calE_p}(G;k)$ is defined ad hoc to extend the field case.

    \begin{question*}
        Can a permutation twisted cohomology theory be expected to describe the geometry of $\Spc(\catK(G;k))$ (without a colimit theorem) when $k$ is a commutative ring, even without a Noetherianity assumption?
    \end{question*}

    \subsection*{Main results} In this paper, we answer these lingering questions in one fell swoop. We describe a \emph{reduced permutation twisted cohomology} construction for any group $G$ and commutative ring $k$ which is only bigraded, twisted by one auspicious endotrivial $x\reg$, which arises from the (doubled) regular real representation $\bbR G$, or more precisely, the associated representation sphere, the one-point compactification. Our main theorem is as follows, and is a summary of results from \Cref{sec:arb_base}.

    \begin{customthm}{A}\label{thm-intro-a}
        Let $G$ be a finite group, $k$ a commutative ring, and $\tH(G;k)$ be the reduced twisted cohomology ring \[\tH(G; k) := \bigoplus_{s \in \bbZ, n \in \bbN} \Hom_{\catK(G;k)}(\bbone, x\reg^{\otimes n}[s]).\] The comparison map
        \[\comp_G\colon \Spc(\catK(G;k)) \to \Spec^h(\tH(G;k))\]
        is an open immersion, restricting to a homeomorphism \[\comp_G\colon \Spc(\catK(G;k)) \xrightarrow{\sim} \Proj(\tH(G;k)).\]
        Moreover, there is an open cover $\{U(H)\}_{H\in \Sub_\Phi(G)}$ of $\Spc(\catK(G;k))$ for which the (untwisted) comparison map restricts to a homeomorphism on each open. In particular, $\Spc(\catK(G;k))$ is a Dirac scheme in the sense of \cite{HP25}. If $k$ is local, then $\{U(H)\}_{H\in \Sub_\Phi(G)}$ separates the closed points of the spectrum.
    \end{customthm}

    We also settle Noetherianity of the reduced twisted cohomology ring when $k$ is Noetherian, and in fact, show that \emph{any} permutation twisted cohomology ring under mild hypotheses is finitely generated. The following is part of the statement \Cref{thm:perm_noeth_thm}.

    \begin{customthm}{B}
        Let $G$ be a finite group, $k$ a commutative Noetherian ring, and $\catM$ be a free submonoid of $\Pic(\catK(G;k))$ generated by effective orientable endotrivials (see \Cref{def-effective-orientable}). The $\catM$-graded permutation twisted cohomology ring \[\tH_{\catM}(G;k) \coloneqq  \bigoplus_{s\in\bbZ}\bigoplus_{x \in \catM}\Hom_{\catK(G;k)}(\bbone, x[s])\] is finitely generated, hence Noetherian.
    \end{customthm}

    In fact, \Cref{thm:perm_noeth_thm} is stronger: it states that $\tH_\catM(G;k)$ is finite over an explicitly-described subalgebra generated by so-called \emph{forerunners morphisms}. This follows from a general criterion for finite generation of twisted cohomology rings in tensor-triangular geometry which we develop in the appendix of this paper.

    \begin{customthm}{C}(\Cref{thm:general_th_thm})
        Let $k$ be a commutative Noetherian ring, and $\catK$ be an essentially small $k$-linear tensor-triangulated category. Let $\catM \leq \Pic(\catK)$ be a free submonoid with generators $x_1, \dots, x_m \in \Pic(\catK)$ and let $\calF = \{f_1, \dots, f_n\} $ consist of homogeneous elements of the $\catM$-graded twisted cohomology ring $\tH_\catM(\catK)$. Assume $\on{H}^\sbull(\catK)$ is a finitely generated $k$-algebra with a finite set of generators belonging to $\calF$, and for all $x \in \catK$, $\Hom_\catK^\sbull(\bbone, x)$ is a finitely generated $\on{H}^\sbull(\catK)$-module.

        Suppose for each $x_i \in \catM$, \[\calF_{x_i}\neq \varnothing, \text{    and   }\bigotimes_{f\in \calF_{x_i}} \cone(f) = 0 \in \catK.\]
        Then for every $y \in \catK$, $\tH_\catM(y)$ is finite over the subalgebra of $\tH_\catM(\catK)$ finitely generated by $\calF$. In particular, $\tH_\catM(\catK)$ is a finitely generated $k$-algebra, hence Noetherian.
    \end{customthm}

    Finally, specializing to the case of $k$ a field of prime characteristic $p$, there is a `geometric open,' the open $U(S)$ for $S \leq G$ a $p$-Sylow. This contrasts the `cohomological open' $U(1)$ whose localization is precisely the derived category $\on{D}_b(kG)$. Balmer--Gallauer showed for an elementary abelian $p$-group $E$, the associated localization $\catK(E;k)|_{U(E)}$ admits a description reminiscent of geometric fixed points. We demonstrate the same holds here, and determine the $U(S)$-local Picard group. The following is a summary of results from \Cref{sec:geo_open}.

    \begin{customthm}{D}
        Let $G$ be a finite group and $k$ a field of prime characteristic $p$. For any Sylow $p$-subgroup $S \in \Syl_p(G)$. the open $U(S) \subseteq \Spc(\catK(G;k))$ is equivalently the points $\catP\in \Spc(\catK(G;k))$ specializing to $\catM(S)$. We have an identification \[\catK(G;k)|_{U(S)} = \frac{\catK(G;k)}{\langle k(G/H) \mid H\in \Sub_p(G)\setminus \Syl_p(G)\rangle}.\] Finally, we deduce the $U(S)$-local Picard group: \[\Pic(\catK(G;k)|_{U(S)}) \cong \bbZ \times \Hom(\Weyl{G}{S}, k^\times).\]
    \end{customthm}

    \subsection*{Why reduced twisted cohomology?} Twisted cohomology theories are of course nothing new, with the notion having existed for nearly a century, beginning with Steenrod, Reidermeister, Eilenberg, and Olum in the context of de Rham cohomology \cite{Ste43, Rei44, Eil47, Olu50}; see also \cite{nlab:twisted_cohomology, Ros25} for detailed overviews. The idea amongst these theories is fairly universal: if the native cohomological setting does not contain enough information, adding twists by `bundles' and considering the associated `sections,' whatever that means in the appropriate context, may recover desired information.

    This paradigm was first exemplified in tensor-triangular geometry as follows. Given an essentially small tensor-triangulated category $\catK$, Balmer \cite{Bal10} described a \emph{comparison map}, a continuous map from the Balmer spectrum to the homogeneous Zariski spectrum of the graded endomorphism ring of the tensor unit (the \emph{cohomology ring} of $\catK$), \[\comp_\catK\colon \Spc(\catK) \to \Spec^h(\End_\catK^\sbull(\bbone)).\] It is well-known that if $R$ is a commutative ring, the  comparison map is a homeomorphism \[\comp_R\colon \Spc(\on{D}_{\on{perf}}(R)) \xrightarrow{\sim} \Spec^h(\End^\sbull(\bbone)) = \Spec(R).\] However, if $X$ is more generally a quasicompact quasiseparated scheme, despite the fact that \[\Spc(\on{D}_{\on{perf}}(X)) \cong X,\] the comparison map need not be a homeomorphism due to a lack of global sections, see \cite[Example 8.3]{Bal10}. Dell'Ambrogio--Stevenson \cite{DS14} developed a general theory of twisted cohomology rings (more precisely, twisted cohomology \emph{2-rings}) and proved that if $X$ has an ample family of line bundles, one can construct an associated twisted cohomology (2-)ring for which the twisted comparison map recovers $\Spc(\on{D}_{\on{perf}}(X))$ openly immersed. However, the twisted comparison map need not be surjective in general.

    In our setting, the usual cohomology ring of the tensor unit is about as trivial as one can expect, given that $\catK(G;k)$ is only a homotopy category. Indeed, while $\catD(G;k)$ admits a t-structure, it does not restrict to the compact part. Therefore, one can try to carry out the same process here and look for an `ample family' of tt-line bundles, a notion which Gallauer \cite{Gal25} has formalized. We wish to find enough endotrivials $x$ admitting global sections $f\colon \bbone \to x$ such that the opens $\open(f) \coloneqq  \supp(\cone(f))^c$ form a basis for the topology of $\Spc(\catK(G;k))$.

    Initially, Miller expected that an ample family of line bundles for $\catK(G;k)$ would need to contain numerous independent twists, and indeed, it is certainly possible to exhibit non-ample families of endotrivials failing to capture enough $p$-local information. For instance, Balmer--Gallauer's original twisted cohomology ring is non-injective outside of the case of elementary abelian $p$-groups. However, the key point that demonstrates otherwise is that our lone twist $x\reg$ both encodes all the $p$-local cohomological data we need (see \Cref{prop:lifting_lemma}) and is able to distinguish between the $p$-local cohomological opens (see \Cref{thm-injectivity-diagonal}). In fact, this twisted cohomology ring does a better job of distinguishing the closed points than Miller's twisted cohomology ring, which we show does not always separate the closed points in \Cref{thm:non_stratify_closeds_ex}. In particular, this gives a negative answer to a separate question of Miller regarding multiplicativity of the forerunner construction.

    Our strategy to prove the homeomorphism is as follows: we first establish Noetherianity of the reduced twisted cohomology ring, \Cref{thm:perm_noeth_thm} when $k$ is Noetherian, and give an open cover of $\Spc(\catK(G;k))$ parametrized by conjugacy classes of the so-called singular subgroups of $G$, \Cref{def-singular-primes}. Singular subgroups were previously used to deduce the classification of endotrivials over a commutative Noetherian ring \cite{GM26}, and in this setting are the appropriate extension of $p$-groups for arbitrary characteristic base rings. Then, additionally assuming $k$ is characteristic $p$ prime, we establish that the comparison map restricts on each open to a homeomorphism via a similar argument to Miller \cite{Mil25c}. From there, we extend to the case of commutative Noetherian base rings via a power-surjectivity Bockstein spectral sequence argument, motivated by van der Kallen and Franjou \cite{vdKF10}. Then, we establish that in all these cases, the comparison map restricts to a homeomorphism on the Proj subvariety. Finally, we lift that result to the case of $k$ non-Noetherian via Noetherian approximation relying on results of Gallauer \cite{Gal18}.

    \subsection*{Organization} We have written this paper with the expected audience not necessarily having a background in at least one of tensor-triangular geometry, permutation modules, or integral representation theory of finite groups. Sections 1, 2, and 3 are mostly preliminary review. Section 1 provides a crash course on tensor-triangular geometry. Section 2 provides a brief overview of the tensor-triangular geometry of permutation modules, both over fields and commutative Noetherian rings. Section 3 introduces the invertible objects of the categories in question, endotrivial complexes, and a crucial related construction, the forerunner morphisms.
    Section 4 introduces a generalized framework for permutation twisted cohomology theories which encompasses previous constructions of permutation twisted cohomology, as well as our new one. This section can be paired with Appendix A, where a general Noetherianity criterion for twisted cohomology rings in tensor-triangular geometry is proved. We also review previous notions of permutation twisted cohomology in this section.

    Section 5 introduces the remixed twisted cohomology ring and verifies standard results not already established in Section 4, such as partial Green functoriality, and the associated open cover. Section 6 proves for $k$ a commutative Noetherian ring that the comparison map is an open immersion. Section 7 proves that the restricted comparison map is a homeomorphism, verifies Miller's remixed twisted cohomology ring does not separate the closed points of the spectrum, and analyzes the geometric open of the spectrum. Finally, Section 8 extends the main result beyond the Noetherian base ring setting.

    \subsection*{Acknowledgments} The authors deeply thank Paul Balmer, Dave Benson, Henning Krause, Martin Gallauer, and Stefan Schwede for numerous enlightening discussions thoughout the course of the project, and additionally thank Dave for remarks on a previous version of this paper. This project was initiated when Miller visited Universität Bielefeld in May, and he thanks the department of mathematics, in particular the Bielefeld Representation Theory Group, for their hospitality. G\'omez is supported by the Deutsche Forschungsgemeinschaft (Project-ID 491392403 – TRR 358). Miller is partially supported by an AMS-Simons travel grant.

    \subsection*{Statement on generative AI use} Most of the research was conducted without the use of generative AI, and in particular, all proofs written in the main paper are human-generated. Moreover, all writing was generated by humans repeatedly tapping keyboard keys. However, the authors utilized ChatGPT 6 Pro to determine a strategy to prove the main result of the appendix, \Cref{thm:general_th_thm}, which plays a critical role in the rest of the paper. We discuss the prompting used: initially, we posed to the AI the question of finite generation of the reduced twisted cohomology ring with the expectation that the forerunners radically generate the entire positive-twist part of the twisted cohomology ring. This expectation was first communicated to the authors by Dave Benson. After a few tries, the AI found a quick proof of the fact relying on the tensor-vanishing condition in \Cref{thm:general_th_thm}. The authors then asked it to attempt to generalize to the full twisted cohomology ring, which it successfully did by inducting on the rank of the Picard group; this led to the strategy of \Cref{lem:inductive_lemma} in full. The authors extended the result to the general tensor-triangulated setting; the proof differs slightly from the AI-generated proof but shares the same underlying idea.

    The authors also tested SAIR's Open Math Model, running DeepSeek V4.1 Flash, to see if it could reproduce the Noetherianity result. With some hints on proof direction, it was able to. The authors additionally had the model, and ChatGPT 6 Pro, audit the paper for correctness and typos (in addition to the authors themselves auditing the entire paper, of course).

    The authors acknowledge interactions with generative AI tools at various points in other related projects which have affected the course of this project. For instance, during the review process of \cite{Mil25c}, ChatGPT 6 Pro discovered an incorrect lemma, which subsequently led to Miller producing (without AI assistance) a more direct argument for proving injectivity of the graded comparison map. A variation of this argument appears again here in \Cref{prop:lifting_lemma} and \Cref{thm:injectivity}.

    \section{Tensor-triangular geometry}

    To begin, we recall some tensor-triangular geometric preliminaries. Primary overviews may be found in \cite{Bal10b, Ste18}.

    \begin{recollection}
        A \emph{tensor-triangulated category} (tt-category for short) is a symmetric monoidal triangulated category with an exact tensor product in each variable separately. A \emph{big tt-category} $\catT$ is a rigidly-compactly generated tt-category; that is, it is compactly generated as a triangulated category and its \emph{compact objects} coincide with its \emph{rigid} (i.e., \emph{dualizable}) objects. Note this is a ``unital algebraic stable homotopy category'' in the sense of \cite{HPS97}.

        We denote by $\catK \coloneqq  \catT^c$ the compact/dualizable part of $\catT$, which is necessarily essentially small. A \emph{tt-functor} is a triangulated functor between tt-categories which is strong monoidal. A tt-functor between big tt-categories is \emph{geometric} if it preserves coproducts. Note that in particular a geometric functor preserves compact objects.

        A full triangulated subcategory $\catI$ of a tt-category $\catK$ is \emph{thick} if it is closed under direct summands, a \emph{thick $\otimes$-ideal} if in addition it has ideal closure ($\catI \otimes\catK \subseteq \catI$), and \emph{prime} if furthermore, it is proper and satisfies the following primality conditions: $x \otimes y \in \catI$ implies $x \in \catI$ or $y \in \catI$. If $\catK$ is essentially small (e.g. $\catK = \catT^c$ for a big tt-category $\catT$), the \emph{Balmer spectrum} $\Spc(\catK)$ (first constructed by Balmer \cite{Bal05}) is the set of prime thick $\otimes$-ideals,
        \[\Spc(\catK) \coloneqq \{\catP \subset \catK \mid \catP \text{ is a prime thick $\otimes$-ideal}\}\]
        endowed with the topology whose basis of \emph{closed} sets is given by supports of objects $x \in \catK$,
        \[\supp(x) \coloneqq  \{\catP \in \Spc(\catK)\mid x \not\in \catP\}.\]
        In fact, $\Spc(\catK)$ is a spectral space; that is, it is homeomorphic to the Zariski spectrum of a commutative ring.
        The Balmer spectrum is contravariantly functorial: given a tt-functor $F\colon \catK\to\catL$, one obtains a continuous map on spectra
        \[\Spc(F)\colon \Spc(\catL) \to \Spc(\catK)\]
        given by $\Spc(F)(\catP)\coloneqq F^{-1}(\catP)$.
    \end{recollection}

    \begin{recollection}
        We next recall the \emph{presheaf of triangulated categories} following \cite{BF07}. This is also due to Balmer (see \cite{Bal02}), and predates the construction of the spectrum \cite{Bal05}. With $\catK$ as before, let $U\subseteq \Spc(\catK)$ be a quasi-compact open subset and let $Z \coloneqq \Spc(\catK) \setminus U$. We define the thick subcategory
        \[\catK_{Z}\coloneqq  \{x \in \catK \mid \supp(x) \subseteq Z\},\]
        i.e., the objects which should vanish on $U$. The category
        \[\catK|_U\coloneqq (\catK/\catK_Z)^\natural\]
        is the idempotent completion of the Verdier quotient $\catK/\catK_Z$. The composition $\catK \twoheadrightarrow\catK/\catK_Z \hookrightarrow \catK|_U$ yields a restriction functor $\catK\to\catK|_U$, and the induced map on Balmer spectra is nothing more than the open immersion $U \hookrightarrow \Spc(\catK)$. This is a genuine presheaf: the restrictions are compatible with intersections, see \cite{BF07} for details.
    \end{recollection}

    \begin{recollection}
        We recall Balmer's \emph{comparison map} \cite{Bal10}. With $\catK$ as before, the graded cohomology ring
        \[
        \End^\sbull_\catK(\bbone)\coloneqq \bigoplus_{n\in \bbZ}\End^n_\catK(\bbone) = \bigoplus_{n\in \bbZ}\Hom_\catK(\bbone, \bbone[n])
        \]
        of $\catK$ is the $\bbZ$-graded endomorphism ring of the monoidal unit $\bbone\in \catK$.
        Note that occasionally it is reasonable to only consider the $\bbN$-graded variant. We have a continuous map, the \emph{comparison map}:
        \begin{align*}
            \comp_\catK\colon \Spc(\catK) &\to \Spec^h(\End_\catK^\sbull(\bbone)) \\
            \catP &\mapsto \langle f \in \End_\catK^\sbull(\bbone) \mid \cone(f) \not\in \catP\rangle.
        \end{align*}
    \end{recollection}

    \begin{example}
        An important example in representation theory (and for the sequel) is the following. Let $G$ be a finite group (scheme) and $k$ a field of characteristic $p$. We set $\on{D}_b(kG) \coloneqq \on{D}_b(\mathsf{mod}(kG))$. Then the comparison map induces a homeomorphism \[V_G \coloneqq  \Spc(\on{D}_b(kG)) \cong \Spec^h(\on{H}^\sbull(G;k))\] which, under the localization $\mathsf{stmod}(kG) \cong \on{D}_b(kG)/\on{D}_{\on{perf}}(kG)$ (see \cite{Ric89}) reduces to a homeomorphism \[\Spc(\mathsf{stmod}(kG)) \cong \Proj(\on{H}^\sbull(G;k)).\] This result is classically due to Benson--Carlson--Rickard \cite{BCR97} for groups and Friedlander--Pevtsova \cite{FP07} for group schemes, although the presentation here is due to Balmer's comparison map framework \cite{Bal10}. This result holds more generally beyond fields, see \Cref{rec-spc-dperm} and also \cite{Lau23,BBIKP25}.
    \end{example}

    \begin{example}
        For a commutative ring $R$, the comparison map induces a homeomorphism
        \[\Spc(\on{D}_{\on{perf}}(R)) \cong \Spec(R),\]
        however, while one also has that $\Spc(\on{D}_{\on{perf}}(X))$ is homeomorphic to the underlying topological space $X$ for any quasi-compact quasi-separated scheme \cite{Tho97}, the comparison map is not a homeomorphism in general; indeed, this already fails for the projective line.

        Dell'Ambrogio--Stevenson remedied the lack of cohomological information by enlarging the cohomology ring. Rather than only considering the graded endomorphism ring of the tensor unit, one may additionally consider \emph{global sections} $\bbone \to x$ for $x$ a line bundle. This produces a multigraded \emph{twisted cohomology ring}, and under suitable hypotheses (namely, the scheme admits an ample family of line bundles), the associated comparison map is a homeomorphism \cite{DS14}. See the introduction for a more detailed recollection.
    \end{example}

    \section{Permutation modules}

    We set $G$ a finite group and $k$ a commutative ring initially. For a complete background on permutation modules from a tt-geometric perspective, see \cite{BG25, BG23b}, and for a modular representation theoretic perspective, see \cite{Las23} or \cite[Chapter 5]{Li18a}.

    \begin{definition}
        A $kG$-module $M$ is:
        \begin{enumerate}
            \item a permutation module if $M$ admits a $G$-stable basis, or equivalently, $M \cong kX$ for some $G$-set $X$;
            \item a $\natural$-permutation module
            if $M$ is a direct summand of a permutation module. If $k$ is a field of characteristic $p$, then equivalently, $\Res^G_S M$ is a permutation module for some Sylow $p$-subgroup $S$ of $G$.
        \end{enumerate}
    \end{definition}

    \begin{recollection}
        We write $\mathsf{perm}(G;k)$ to denote the additive category of finitely generated permutation modules, and $\mathsf{perm}(G;k)^\natural$ for its idempotent completion. We set \[\catK(G;k) \coloneqq \on{K}_b(\mathsf{perm}(G;k)^\natural) \cong \on{K}_b(\mathsf{perm}(G;k))^\natural \] and \[\catD(G;k) \coloneqq  \on{DPerm}(G;k),\] where $\catD(G;k)$ is the localizing subcategory of $\mathsf{K}(\mathsf{Mod}(kG))$ generated by $\catK(G;k)$.

        Then $\catD(G;k)$ is a rigidly-compactly generated tt-category with rigid-compact part $\catK(G;k) = \catD(G;k)^c$. These tt-categories can be defined in the context of cohomological Mackey functors, Artin motives, or equivariant stable homotopy theory, see \cite{BG23b, F25} for details.
    \end{recollection}

   \subsection{Finite permutation resolutions}
   For this subsection, we assume that $k$ is a commutative Noetherian ring.

   \begin{recollection}\label{rec-spc-dperm}
       Let $\on{D}_b(G;k)$ denote the bounded derived category of finitely generated $kG$-lattices, i.e., finite $kG$-modules that are projective as $k$-modules. This is a rigid tt-category, and a thick subcategory of $\on{D}_b(kG)$. The two agree when the ring $k$ is regular, see \cite{Lau23} and the references therein for further details.

       We let $\on{D}_{\perm}(G;k)$ denote the thick subcategory of $\on{D}_b(G;k)$ generated by permutation modules. It is in fact a tt-subcategory. Again, these categories agree provided that $k$ is regular. On the other hand, by \cite{Lau23}, the comparison map
       \[
            \comp\colon  \Spc(\on{D}_b(G;k))\to \Spec^h(\on{H}^\sbull(G;k))
       \]
        is a homeomorphism. From here, we deduce that
        \[
            \comp'\colon  \Spc(\on{D}_{\perm}(G;k))\to \Spec^h(\on{H}^\sbull(G;k))
       \]
       is a homeomorphism as well. Indeed, using that $\on{H}^\sbull(G;k)$ is Noetherian \cite{Ven59,Eve61}, the naturality of the comparison map, and the fact that fully faithful functors of tt-categories induce surjections on Balmer spectra \cite{Bal18}, we get that $\comp'$ is a bijection. The fact that it's a homeomorphism is a consequence of \cite[Corollary 2.8]{Lau23}. Further details are given in \cite[Recollection 5.3]{DG25}.
   \end{recollection}

   \begin{recollection}\label{rec-localization-Upsilon}
         By \cite{BG23a}, every finitely generated $kG$-lattice is a retract of a finitely generated one that has a finite resolution by $\natural$-permutation modules, and consequently, one has a localization
         \[
            \catK(G;k) \twoheadrightarrow \left(\catK(G;k)/\catK_{\on{ac}}(G;k)\right)^\natural \cong \on{D}_{\on{perm}}(G;k)
         \]
         where $\catK_{\on{ac}}(G;k)$ denotes the thick $\otimes$-ideal of $\catK(G;k)$ consisting of all acyclic complexes. This is the compact part of a finite localization on $\catD(G;k)$. See \cite[Theorem 1.3]{BG23b}.

         \cite[Proposition 3.22]{BG25} asserts that this functor induces an open inclusion $\Spec^h(\on{H}^\sbull(G;k)) \hookrightarrow \Spc(\catK(G;k))$ when $k$ is a field, and \cite[Corollary 4.14]{DG25} in the general case. We denote this subset by $V_G \subseteq \Spc(\catK(G;k))$; it is called the \textit{cohomological open}, being quite literally Carlson's cohomological support variety when working over fields of positive characteristic (plus the very closed point, the irrelevant ideal).
    \end{recollection}

    \subsection{Permutation modules over fields}

    For this subsection, we assume $k$ is a field of positive characteristic $p$. Generally we assume $|G|$ is non-invertible in $k$, as otherwise, $\mathsf{perm}(G;k)^\natural = \mathsf{mod}(kG)$ is a semisimple abelian category.

    \begin{recollection}\label{thm:conservative}
        Let $H \in \Sub_p(G)$. \emph{Brauer quotients}, due to Brou\'e \cite{Bro79}, are additive functors on $kG$-modules which mimic taking $H$-fixed points of $G$-sets on permutation modules. When restricted to $p$-permutation modules, these functors are monoidal.

        Balmer--Gallauer adapted this construction to work for $\catD(G;k)$, dubbing them \textit{modular fixed points} functors.
        By \cite[Proposition 2.7]{BG25}, we have the following: for every $p$-subgroup $H \leq G$ there exists a coproduct-preserving tt-functor on the big derived category of permutation modules \[\Psi^H = \Psi^H_G\colon  \catD(G;k) \to \catD(\Weyl{G}{H};k)\] such that $\Psi^H(kX) \cong k(X^H)$ for every $G$-set $X$. In particular, this functor preserves compacts and restricts to a tt-functor $\Psi^H\colon \catK(G;k) \to \catK(\Weyl{G}{H};k)$. We recall the construction explicitly in \Cref{rec:mod_fixed_pts_const} in a more general context.

        By \cite[Theorem 2.11]{BG25}, the family of functors \[\{\check\Psi^H\colon\catD(G;k) \xrightarrow{\Psi^H} \catD(\Weyl{G}{H};k) \twoheadrightarrow \on{K}(\on{Inj}(k(\Weyl{G}{H})))\}_{H \in \Sub_p(G)/G}\] indexed by conjugacy classes of $p$-subgroups $H \leq G$ is conservative, i.e., detects vanishing of objects. This restricts to a conservative family of functors  \[\{\check\Psi^H\colon \catK(G;k) \xrightarrow{\Psi^H} \catK(\Weyl{G}{H};k) \twoheadrightarrow \on{D}_b(k(\Weyl{G}{H}))\}_{H \in \Sub_p(G)/G}\] on the compact parts.
    \end{recollection}

    \begin{notation}
        Each tt-functor $\Psi^H$ induces a continuous map on spectra \[\psi^H\coloneqq  \Spc(\Psi^H) \colon  \Spc(\catK(\Weyl{G}{H};k)) \to \Spc(\catK(G;k)),\] and composing with the surjection $  \catK(G;k) \twoheadrightarrow \on{D}_b(kG)$ induces another continuous map on spectra \[\check{\psi}^H\colon  V_{\Weyl{G}{H}} \to \Spc(\catK(G;k)).\]
    \end{notation}

    \begin{recollection}
        By \cite[Theorem 2.10]{BG25}, every point of $\Spc(\catK(G;k))$ is the image $\check\psi^H(\frakp)$ of a point $\frakp \in V_{\Weyl{G}{H}}$ for some $p$-subgroup $H \leq G$ unique up to $G$-conjugation, i.e., $\check\psi^H(\frakp) = \check\psi^{H'}(\frakp')$ if and only if there exists a $g \in G$ such that ${}^gH = H'$ and ${}^g\frakp = \frakp'$. In other words, the spectrum as a set has the following description: \[\Spc(\catK(G;k)) = \bigsqcup_{H \in \Sub_p(G)/G} \Spc(\on{D}_b(k(\Weyl{G}{H}))) \cong \bigsqcup_{H \in \Sub_p(G)/G} \Spec^h(\on{H}^\sbull(\Weyl{G}{H};k)).\]
    \end{recollection}

    \begin{notation}\label{not:primes}
        We denote by $\catP(H,\frakp) = \catP_G(H, \frakp) \in \Spc(\catK(G;k))$ the image of the prime $\frakp \in V_{\Weyl{G}{H}}$ under $\check{\psi}^H: V_{\Weyl{G}{H}} \to \Spc(\catK(G;k))$. By the previous theorem, every prime is up, to $G$-conjugacy, uniquely expressible in this way.
    \end{notation}

    \begin{remark}
        The subset $\im (\check\psi^H)=\{\catP(H, \frakp) \mid \frakp \in V_{\Weyl{G}{H}}\} \subseteq \Spc(\catK(G;k))$ is generally not open nor closed when $H > 1$. Explicitly, one may have inclusion of primes $\catP(H, \frakp) \subseteq  \catP(K, \frakq)$ when $H >_G K$; see \cite[Corollary 7.13]{BG25}. However, each subset $\im (\check\psi^H)$ contains the \emph{`very closed point'} \[\catM(H) \coloneqq  \catP(H,0) \in \Spc(\catK(G;k)).\] All closed points of $\Spc(\catK(G;k))$ are of this form.
    \end{remark}

    \begin{remark}
        If $k$ is a field with $|G|$ invertible in $k$ (for instance, if $k$ has characteristic 0), then since $\mathsf{perm}(G;k)^\natural = \mathsf{mod}(kG)$ is a semisimple abelian category, we have $\catK(G;k) = \on{D}_b(kG)$, and $\check\Psi^1$ is simply the identity functor.
    \end{remark}

    \subsection{Integral permutation modules} Next, we describe how one can get a handle on permutation modules with coefficients in a commutative Noetherian ring. This theory was developed in \cite{Gom25,DG25}.

   \begin{recollection}\label{rec-setbijection-spc}
       Let $\frakp$ be a prime of $k$. We write $\lambda_\frakp \colon \catD(G;k)\to \catD(G;k(\frakp))$ to denote the geometric functor induced by base change along the fiber $k\to k(\frakp)$. By \cite[Theorem 4.8]{Gom25}, the family of functors $\{\lambda_\frakp \mid \frakp\in \Spec(k)\}$ is jointly conservative, and by \cite[Theorem 6.18]{Gom25}, this family gives us a set-theoretic bijection
       \[
       \bigsqcup_{\frakp\in \Spec(k)}\Spc(\catK(G;k(\frakp)))\xrightarrow{\simeq} \Spc(\catK(G;k)).
       \]
    Moreover, the space $\Spc(\catK(G;k))$ is Noetherian by \cite[Corollary 6.12]{Gom25}.
   \end{recollection}

   \begin{notation}
       Given $\frakp \in \Spec(k)$, $H \in \Sub_{\on{char}(k(\frakp))}(G)$ (with the convention that if $\on{char}(k(\frakp))$ does not divide $|G|$, then $H=1$), and $\fraka \in \Spc(k(\frakp)(\Weyl{G}{H}))$, we let $\catP(\frakp, H, \fraka)\in \Spc(\catK(G;k))$ denote the image of the prime $\catP(H, \fraka) \in \Spc(\catK(G;k(\frakp)))$. We set $\catM(\frakp; H) \coloneqq \catP(\frakp, H, 0)$, which corresponds to the very closed point $\catM(H) \in \Spc(\catK(G; k(\frakp)))$.
   \end{notation}

  \begin{remark}
    Note that if $\on{char}(k(\frakp))$ does not divide $|G|$, then $\catP(\frakp, 1, \fraka)$ is completely determined by $\frakp$. In fact, it goes to $\frakp$ under the comparison map
    \[
    \Spc(\catK(G;k))\to \Spec(k).
    \]
   \end{remark}

   \begin{recollection}
       Set $n=|G|$. Then base change along $k\to k[1/n]$ induces an open immersion
       \[
        \Spc(\catK(G;k[1/n]))\cong \Spec(k[1/n])\to \Spc(\catK(G;k))
       \]
       whose image identifies with the open
       \[
       (\Spec(\eta)\circ \mathrm{comp})^{-1}(\Spec(k[1/n]))\subseteq \Spc(\catK(G;R))
       \]
       where $\eta$ denotes the structural map $\mathbb Z\to k$; see \cite[Proposition 6.5]{Gom25} and \cite[Theorem 2.11]{GM26}.  We baptize this open as the \textit{ordinary locus} of $\catK(G;k)$, and its set-theoretic complement is the \textit{modular locus}. Hence one is left to understand how this two pieces glue together. In fact,  one gets that the closure of the ordinary locus lies in the cohomological open. Indeed, this has been showed in \cite[Theorem 6.1]{DG25} for $p$-groups, but the same argument works for general finite subgroups using the joint conservativity of the base change functors from above. Hence it is enough to understand the modular locus. For $p$-groups, and rings such that $p$ is not a zero-divisor, we have an answer: reduction modulo $p$ induces a closed embedding
       \[
       \Spc(\catK(G;k/pk))\to \Spc(\catK(G;k)).
       \]
       See \cite[Theorem 6.4]{DG25}. However, we do not know if this holds more generally.
   \end{recollection}

    \begin{proposition}
        \label{prop-reduction-modp}
        Let $p$ be a prime dividing the order of $G$ and such that $pk\not=k$. Let $\lambda_p\colon \catK(G;k)\to \catK(G;k/pk)$ be the tt-functor induced by reduction modulo $p$. Then $\Spc(\lambda_p)$ is injective and its image identifies with $(\eta\circ \comp)^{-1}(p)$.
    \end{proposition}
    \begin{proof}
        This follows from \Cref{rec-setbijection-spc}, the naturality of the comparison map, and the fact that any residue field of $k$ of characteristic $p$ factors through $k/pk$.
    \end{proof}

    \subsection{Generalized modular fixed points}\label{subsec-generalized-modular}
    Let us record a construction from \cite[Section 4]{GM26}, which essentially is Balmer--Gallauer's modular fixed points construction (see \Cref{thm:conservative}) over an arbitrary ring of characteristic $p$.

    \begin{recollection}\label{rec:mod_fixed_pts_const}
        Let $k$ be a commutative ring of characteristic a prime $p$, and let $N \trianglelefteq G$ be a normal $p$-subgroup. Let $\mathcal{F}_N=\{H\leq G\mid N\not\leq H\}$, and set
        \[\mathsf{proj}(\mathscr{F}_N;k) \coloneqq \mathsf{add}^\natural\{k(G/K)\mid K \in \catF_N\},\]
        the thick additive closure of $\{k(G/K)\mid K \in \catF_N\}$ in $\mathsf{perm}(G;k)^\natural$. Then $\mathsf{proj}(\mathscr{F}_N;k)$ is a $\otimes$-ideal of $\mathsf{perm}(G;k)^\natural$, and the composite \[\mathsf{perm}(G/N; k)^\natural \xrightarrow{\Inf^G_{G/N}} \mathsf{perm}(G;k)^\natural \twoheadrightarrow \frac{\mathsf{perm}(G; k)^\natural}{\mathsf{proj}(\mathscr{F}_N;k)}\] is an equivalence of tensor categories. We get a tt-functor
        \[
        \Psi^N\colon \catK(G;k) \to \catK(G/N;k)
        \]
        by extending the functor $\Psi^N\colon \mathsf{perm}(G;k)^\natural \twoheadrightarrow \frac{\mathsf{perm}(G; k)^\natural}{\mathsf{proj}(\catF_N;k)} \cong \mathsf{perm}(G/N; k)^\natural$ degreewise.

        For an arbitrary $p$-subgroup $H\leq G$, we define the \emph{modular $H$-fixed points functor} by the composite \[\Psi^H\colon \catK(G;k) \xrightarrow{\Res^G_{N_G(H)}} \catK(N_G(H);k) \xrightarrow{\Psi^H} \catK(\Weyl{G}{H};k).\] If $k$ is a field of characteristic $p$, this is precisely Balmer--Gallauer's construction.
    \end{recollection}

    We now explain how to extend this to general commutative rings. We need the following definition which corresponds to \cite[Definition 4.3]{GM26}.

    \begin{definition}\label{def-singular-primes}
        Let $\Phi(G;k)$ denote the set of primes $p$ for which $p \mid |G|$ and $pk \neq k$, i.e., $p \not\in k^\times$. We also always include $0 \in \Phi(G;k)$. By a 0-subgroup of $G$, we mean the trivial subgroup $1$ of $G$.

        We call the elements of $\Phi(G;k)$ the \emph{singular primes} of $k$ with respect to $G$. For example, if $k$ is a field of characteristic $p$ and $p \mid |G|$, then $\Phi(G;k) = \{p\}$.

        We write $\Sub_\Phi(G;k)$ to denote the family of $p$-subgroups of $G$, where $p$ runs over $\Phi(G;k)$.  We call these subgroups the \emph{singular subgroups} of $G$ (with respect to $k$).
    \end{definition}

    \begin{definition}\label{def:genl_mod_fixed_pts}
        Let $p \in \Phi(G;k)$. Given a $p$-subgroup $H \leq G$, we define the \emph{generalized $H$-modular fixed points functor} by \[\Psi^H_p \coloneqq \Psi^H\circ\lambda_p \colon \catK(G;k) \to \catK(\Weyl{G}{H};k/pk)\]
        where $\lambda_p$ is the functor induced by reduction modulo $p$. If $k$ is a ring of characteristic $p$, this is simply the usual modular fixed points functor $\Psi^H$.
    \end{definition}

    \begin{recollection}
        The construction of generalized modular fixed points is compatible with base chance to residue fields. More generally,  let $k, k'$ be commutative rings, $p \in \Phi(G;k) \cap \Phi(G;k')$ be prime, let $H \leq G$ be a $p$-subgroup, and let $f\colon k\to k'$ be a ring homomorphism. Let $f_p \colon k/pk \to k'/pk'$ be the induced ring homomorphism. Then for all $p$-subgroups $H\leq G$, the following diagram commutes.
        \begin{figure}[H]
            \centering
            \begin{tikzcd}
                \catK(G;k) \ar[r, "\Psi^H_p"] \ar[d, "f^\ast "]& \catK(\Weyl{G}{H}; k/pk) \ar[d, "f_p^\ast "] \\
                \catK(G; k') \ar[r, "\Psi^H_p"] & \catK(\Weyl{G}{H}; k'/pk')
            \end{tikzcd}
        \end{figure}
    \end{recollection}

    \section{Endotrivial complexes and forerunners}

    We continue by reviewing endotrivial complexes. For this section, we assume $k$ is a commutative ring.

    \begin{definition}
        An \emph{endotrivial complex} is an invertible object of $\catK(G;k)$, i.e., a chain complex $x \in \catK(G;k)$ for which $x^* \otimes x = \mathsf{hom}(x,x) \simeq \bbone$. The Picard group $\Pic(\catK(G;k))$ is the group of endotrivials under tensor product.
    \end{definition}

    The study of the endotrivials began with Bachmann's thesis \cite{Bac16} in the motivic setting, then later by Miller \cite{Mil24a, Mil25a}, who was at the time unaware of Bachmann's prior work and the link between permutation modules and Artin motives. The Picard group $\Pic(\catK(G;k))$ for $k$ a field of positive characteristic was determined for any finite group by Miller \cite{Mil25b}, and for $k$ a commutative Noetherian ring by the authors \cite{GM26}.

    \begin{notation}
        We denote the set of $p$-valued superclass functions by $\on{CF}(G,p)$; i.e., functions $\Sub_p(G)/G \to \bbZ$.
    \end{notation}

    \begin{definition}\label{def:h_mks}
        Assume $k$ is a field of prime characteristic $p$. Given an endotrivial $x \in \Pic(\catK(G;k))$, we construct a $p$-valued superclass function $h_x\in \on{CF}(G,p)$, termed the \emph{h-mark function} as follows: for $H \in \Sub_p(G)$, $h_x(H)$ is the unique homological degree for which $\Psi^H(x)$ is non-exact. The values $h_x$ are the \emph{h-marks} of $x$; this can be thought of as a virtual `dimension function' of $x$.

        One obtains a group homomorphism \[h\colon \Pic(\catK(G;k)) \to \on{CF}(G,p), \quad x \mapsto h_x,\] which we term the \emph{h-mark homomorphism}.
    \end{definition}

    It turns out that the h-marks of an endotrivial complex almost completely determine its isomorphism type. The \emph{Borel--Smith conditions} play a significant role in the classification of endotrivials; see \cite[Borel--Smith functions]{tD87} or \cite{DH81} for their explicit definition. For the context of this paper, we will never require the formal definition, so we omit it for brevity.

    \begin{recollection}
        Let $G$ be a finite group and $k$ a field of prime characteristic $p$. By \cite{Mil25b}, we have a split exact sequence \[0 \to \Hom(G,k^\times) \to \Pic(\catK(G;k)) \xrightarrow{h} \on{CF}_b(G,p)\to 0,\] where $\on{CF}_b(G,p)$ denotes the subgroup of Borel--Smith superclass functions.
    \end{recollection}

    The classification over a commutative Noetherian ring exhibits a similar exact sequence arising from \emph{fiberwise h-marks}; see \cite{GM26} for details.

    \begin{recollection}
        Let $k$ be a commutative ring. Given a representation sphere $S^V$ (i.e., the one point compactification of a real representation $V$) with finite $G$-CW structure induced from a triangulation of $S^V$ (this is always possible by \cite{Ill83}), its reduced cellular chain complex is endotrivial, and its h-mark function is precisely the dimension function of $S^V$.

        More formally, the reduced cellular chain complex $\tilde{C}((S^V)^?;k)$ is a bounded chain complex of finitely generated free right $k\Gamma_G$-modules, where $\Gamma_G$ denotes the orbit category over $G$.
        Then evaluation at the trivial subgroup 1 produces a bounded chain complex of permutation $kG$-modules (where the $G$-action is induced from $\Aut_{\Gamma_G}(G/1) \cong G$). We set $x_V \coloneqq  \tilde{C}((S^V)^1;k)$. See e.g. \cite{HY14} for details.

        Let $\on{RO}^+(G)\subseteq \on{RO}(G)$ denote the submonoid of real representations (as opposed to virtual real representations). The assignment \[\on{RO}^+(G)\ni V \mapsto \tilde{C}((S^V)^1;k) \in \catK(G;k)\] induces a group homomorphism \[\catB\colon \on{RO}(G) \to \Pic(\catK(G;k)).\] Indeed, $\catB(S^{V\oplus W}) \cong \catB(S^{V} \wedge S^W) \cong \catB(S^V) \otimes \catB(S^W)$, and additive inverses in $\on{RO}(G)$ correspond to duals in $\Pic(\catK(G;k))$. Moreover, $\catB$ is surjective when $k$ is a field of prime characteristic $p$ and $G$ is a $p$-group by \cite{DH81} and \cite{Mil25b}, since the h-mark homomorphism $h\colon \Pic(\catK(G;k)) \to \on{CF}_b(G)$ is an isomorphism.

        When $G$ is non-nilpotent, $\catB$ is not necessarily surjective, even over fields, as endotrivials in $\im(\catB)$ have h-marks satisfying the \emph{Artin condition} introduced by Bauer \cite{Bau89}; endotrivials in general need not satisfy this condition. Indeed, more generally any invertible genuine $G$-spectra has dimension function satisfying the Artin condition, and $\catB$ sends dimension functions to h-mark functions.
    \end{recollection}

    \begin{definition}\label{def-effective-orientable}
        We say any endotrivial which up to a degree shift belongs to $ \catB(\on{RO}^+(G))$ is \emph{effective} and say any endotrivial belonging to $\im(\catB)$ is \emph{quasieffective}. In particular, we allow for negative degree shift for effective endotrivials. When $G$ is a $p$-group, $\Pic(\catK(G;k))$ has a canonical basis of effective endotrivials arising from $\Irr_\bbR(G)$. See \cite{Mil25c} for an overview.

        The following definition arises from tom Dieck--Petrie's study of homotopy representations \cite{tDP82}. A representation sphere is \emph{orientable} if for all $H \in \Sub(G)$ and $gH \in \Weyl{G}{H}$, the action of $gH$ on $S^{V^H}$ is orientation-preserving, or equivalently $\det(\rho_{V^H}(gH)) = 1$.

        Recall the definition of a singular subgroup, see \Cref{def-singular-primes}. The following definition is motivated by the fact that we can only see $\Phi$-locally in $\catK(G;k)$. The representation sphere $S^V$ is \emph{orientable modulo $k$} if for all $H \in \Sub_\Phi(G;k)$, the nonzero homology of $\tilde{C}((S^V)^H;k)$ has trivial $\Weyl{G}{H}$-action. We say a quasieffective endotrivial $x$ is orientable modulo $k$ if there exist real representations $V_1, V_2$ satisfying $ x= x_{V_1}\otimes x_{v_2}^*$ and both $S^{V_1}$ and $S^{V_2}$ are.

        More generally, we say an endotrivial $x \in \Pic(\catK(G;k))$ is \emph{quasi-orientable modulo $k$} if for all $H \in \Sub_\Phi(G;k)$, the nonzero homology of $\Psi^H_p(x)$ (see \Cref{def:genl_mod_fixed_pts}) has trivial $\Weyl{G}{H}$-action. This is a weaker condition in general. Moreover, the nonzero homology of $\Psi^H_p(x)$ need not be contained in one homological degree, see \cite[Section 6]{GM26}.

        We let $\Pic^{\on{eo}}(\catK(G;k)) \leq \Pic(\catK(G;k))$ denote the submonoid of effective endotrivials orientable modulo $k$.
    \end{definition}

    \begin{remark}\label{rmk:extended_h_mks}
        For quasieffective endotrivials $x$, its h-mark function $h_x$ after base-change to a fiber of characteristic $p\in \Phi(G;k)$ is simply the restriction of the dimension function of an associated virtual representation sphere in the preimage $\catB\inv(x)$. In particular, this holds for all fibers $k(\frakp)$ of characteristic $p\in\Phi(G;k)$. Therefore, one may extend their h-mark function of a quasieffective endotrivial to a h-mark function $h_x \colon \Sub_\Phi(G;k) \to \bbZ$ by base changing to fibers. See \cite[Section 5]{GM26} for a more general treatment.

        We note `fiber-local' h-marks can differ for non-quasieffective endotrivials and sufficiently complicated rings $k$; see \cite[Section 6]{GM26}.
    \end{remark}

    \begin{remark}
        If $G$ is a $p$-group and $k$ a field of characteristic $p$, all representation spheres and endotrivials are automatically orientable modulo $k$. However, for even-order groups, representation spheres need not be orientable in general; in fact, most spheres arising from nontrivial irreducible representations $V$ of real type (i.e., $\bbC \otimes_\bbR V$ remains irreducible) are not orientable. The question of describing the orientable representation spheres in general seems to have not been studied classically.

        Miller considered the question recently in \cite{Mil26d} and was able to describe them for groups with normal Sylow 2-subgroups: in this setting, he deduced a representation sphere associated to an irreducible real representation is non-oriented if and only if the irreducible real representation is of real type (i.e., $\bbC \otimes_\bbR V$ does not decompose) and is nontrivial.
        However, for a general finite group, the question appears difficult.
    \end{remark}

    The following orientable effective endotrivial will be the main player of our story.

    \begin{definition}\label{def:reg_orientable_endotriv}
        We set \[x\reg = x\reg(G) \coloneqq  \catB(\bbR G^{\oplus 2}) = \tilde{C}(S^{{\bbR G}^{\oplus 2}};k).\]
        Then $x\reg$ is automatically orientable regardless of $k$. Notice that this construction is compatible with base change, i.e., we have a natural isomorphism $\lambda_\frakp(x\reg(k)) \cong x\reg(k(\frakp))$, hence we omit the base ring in the notation. We call $x\reg(G)$ the \emph{regular (orientable) endotrivial} for $G$. We write $x\reg$ if the underlying group $G$ is clear. It is an easy exercise to verify \[h_{x\reg(G)}(H) = 2[G:H]. \]
    \end{definition}

    \begin{remark}
        Of course, there exist instances when the endotrivial $\catB(\bbR G)$ is already orientable for $p$ odd (for instance, if $|G|$ is odd), but we keep this convention for convenience.
    \end{remark}

    \subsection{Forerunners and tt-line bundles} We now recall the construction of the so-called \textit{forerunner morphisms}. We first set some terminology.

    \begin{definition}
        A \emph{global section (of degree $s$)} of an object $x\in \catK(G;k)$ (in the sense of Gallauer \cite{Gal25}) is a homomorphism $\bbone \to x[s]$.
    \end{definition}

    The following examples of global sections control much of the geometry of $\catK(G;k)$. These were constructed by Miller for $p$-groups in \cite{Mil25c} to generalize Balmer--Gallauer's $a_N, b_N$ homomorphisms in \cite[Section 12]{BG25}, then extended to all finite groups by us in \cite{GM26}. We note Balmer--Bouc--Dell'Ambrogio have also produced a purely algebraic construction, while ours relies on topological properties.

    \begin{construction}{\cite[Construction 8.5]{GM26}}\label{cons:forerunners}
        Assume $k$ is a commutative ring. Let $S^V$ be a representation sphere orientable modulo $k$, endowed with the structure of a finite $G$-CW-complex (by \cite{Ill83}, for instance by taking the barycentric subdivision of a triangulation), and let $H \leq G$ be a singular subgroup. Fix $x \coloneqq  \tilde{C}(S^V;k)$. We have an obvious identification $\tilde{C}(S^{V^H});k) = \tilde{C}((S^V)^H;k)$. First, note that if $\dim_V(H)=n$, then for all $i > n$, $\tilde{C}_i((S^V)^1;k))$ contains no transitive permutation modules with isotropy subgroup containing $H$, since $S^{V^H}$ is $n$-dimensional.

        Next, in degree $n = \dim_V(H)$, all transitive permutation module summands of $\tilde{C}_n(S^V;k)$ with isotropy subgroups containing up to conjugacy $H$ are in fact isomorphic to $k(G/K)$ for $K \geq H$ a maximal subgroup containing $H$ with respect to satisfying the equality $\dim_V(H) =\dim_V(K)$. Indeed, we have $V^H = V^K$ as vector spaces, hence all $H$-stable cells are $K$-stable but not $K'$-stable for any $K' > K$.

        If $H = K$, i.e., $H$ satisfies $\dim_V(H) > \dim_V(H')$ for all $H' > H$, then we construct the forerunner at $H$ as follows. We have $\tilde{C}_n((S^V)^H;k) \cong k(N_G(H)/H)^{\oplus i}$ for some $i \in \bbN$. Moreover, since $(S^V)^H$ is again a representation sphere with orientation-preserving $\Weyl{G}{H}$-action, we have an isomorphism $\tilde{\iota}_x^H\colon k \cong \ker(d^H_n)$. Then $\tilde{C}_n(S^V;k) \cong k(G/H)^{\oplus i} \oplus k(Y)$ where $Y$ is a (possibly empty) $G$-set satisfying $Y^H = \varnothing$. We define $\iota^H_x\colon k \to x[-n]$ to be the trace sum \[\iota^H_x = \sum_{g\in [G/N_G(H)]} {}^g\tilde{\iota}^H_x\] in homological degree 0, and zero in all other degrees.

        One must check that this is a well-defined chain complex homomorphism, that is, $ \im(\iota^H_x)\subseteq \ker(d^1_n)$. This follows from the fact that $N_G(H)/H = (G/H)^H\subseteq G/H$ satisfies
        \[
        G/H = \bigsqcup_{g\in [G/N_G(H)]} g\cdot (N_G(H)/H),
        \]
        and each component of the trace sum lands in a different disjoint part. Then, the trace sum, $G$-linearity of $d_n^1$, and orientability of $S^V$ together imply the image of $\iota^H_x$ lies in the kernel of $d_n^1$ and is isomorphic to $k$.

        For the case where $K > H$ and $\dim_V(H) = \dim_V(K)$, we define $\iota^x_H \coloneqq  \iota^x_K$, using the above construction. We previously observed that since $V^H = V^K$, we have an equality of chain complexes of $k$-modules $\tilde{C}((S^V)^H;k) = \tilde{C}(S(V)^K;k)$.
    \end{construction}

    These maps satisfy the following properties, which allows us to `detect' local cohomology in $\catK(G;k)$.

    \begin{proposition}\label{prop:forerunner_props}
        Let $S^V$ be a representation sphere orientable modulo $k$, set $x \coloneqq  \catB(S^V)$, let $p \in \Phi(G;k)$ and let $K, H \in \Sub_p(G)$. The following properties hold:
        \begin{enumerate}
            \item $\Psi^H_p(\iota^H_x)$ is a quasi-isomorphism.
            \item If $K =_G H$ are conjugate, then $\iota^H_x = \iota^K_x$.
            \item If $K\neq_G H$, $h_x(H)= h_x(K)$, and all minimal subgroups $L$ with respect to the property that $K,H \leq_G L$ satisfy $h_x(H)= h_x(K) > h_x(L)$, then $\Psi_p^K(\iota_x^H) = 0$ and $\Psi_p^H(\iota_x^K) = 0$. In particular, $\iota^K_x$ and $\iota^H_x$ differ.
            \item If $K \in \Sub_p(G)$ satisfies $h_x(H) \neq h_x(K)$, then $\iota^K_x$ and $\iota^H_x$ differ. In particular, $\Psi^K_p(\iota^H_x)$ and $\Psi_p^H(\iota^K_x)$ are not quasi-isomorphisms.
        \end{enumerate}
    \end{proposition}
    \begin{proof}
        (a)--(c) are shown in \cite[Proposition 8.8]{GM26}. (d) is obvious, as the codomain of $\iota^H_x$ is $x[-h_x(H)]$ and the codomain of $\iota^K_x$ is $x[-h_x(K)]$.
    \end{proof}

    \begin{remark}
        Presently, the forerunner construction relies on a choice of cellularization of a representation sphere $S^V$. We expect however that this choice does not matter in principle; any two `reasonable' cellularizations of $S^V$ should produce the same morphisms (up to a nonzero scalar) in $\Hom_{\catK(G;k)}(\bbone, \catB(S^V)[-\dim V^H])$.
    \end{remark}

    \begin{remark}\label{rem-basechange-forerunners}
        Forerunners are compatible with base change in the following sense. Let $p \in \Phi(G;k)$ be prime, $\frakp \in \Spec(k)$ have fiber $k(\frakp)$ of characteristic $p$, and $H \in \Sub_p(G)$. Then for all effective orientable endotrivials $x$, we have $\lambda_\frakp(\iota^H_{x}) = \iota^H_{\lambda_\frakp(x)}$, as long as one specifies the same triangulation for $x$ in both forerunner constructions, since the construction is entirely independent of base ring. Indeed, at least if $x$ arises from a representation sphere orientable modulo $\bbZ$, then the construction can be performed with coefficients in $\bbZ$, then extending scalars.
    \end{remark}

    Finally, we recall that the orientable endotrivials arising from representation spheres satisfy a highly-desirable property for twisted cohomology: they are locally invertible. This fact arises from the forerunners construction.

    \begin{definition}
        A \emph{tt-line bundle} is an invertible object $x \in \Pic(\catK)$ in a tt-category $\catK$ for which the following holds: there exists a quasi-compact open cover $\{U_i\}_{i \in I}$ of $\Spc(\catK)$ such that $x$ is locally trivial, i.e., isomorphic to a shift of the tensor unit, in each localization $\catK|_{U_i}$.
    \end{definition}

    For our twisted cohomology theory, we wish to consider twists by tt-line bundles, as the associated opens give us an open cover of the Balmer spectrum. Miller showed in \cite{Mil25c} that if $G$ is a $p$-group and $k$ a field of characteristic $p$, then every endotrivial is a tt-line bundle; this is demonstrated by inverting each forerunner. We prove this fact more generally, see \Cref{thm:quasieffects_linebundles}.
    More generally, the authors extended the result for $p$-groups and commutative Noetherian rings by removing the `effective' qualifier.

    \begin{theorem}\cite[Theorem 8.11]{GM26}
        Let $G$ be a finite $p$-group and $k$ a commutative Noetherian ring. Then every quasi-orientable endotrivial $x \in \Pic(\catK(G;k))$ is a tt-line bundle.
    \end{theorem}

    \section{Generalized permutation twisted cohomology}

    We give a general framework for twisted cohomology rings for permutation modules which encompasses Balmer--Gallauer's original twisted cohomology ring \cite{BG25}, Miller's `remixed' version \cite{Mil25c}, and our reduced twisted cohomology ring in the sequel. Our approach fits into our general treatment of twisted cohomology for tt-categories graded over a free monoid of finite rank; see \Cref{def:most_general_twisted_coh}. For this section, we assume $k$ is a commutative ring, sometimes assumed to be Noetherian.

    \begin{definition}\label{def-Mgraded-ptcr}
        Let $G$ be a finite group and $k$ a commutative Noetherian ring. Suppose $\catM \leq \Pic^{\on{eo}}(\catK(G;k))$ (see \Cref{def-effective-orientable}) is a free submonoid of effective endotrivials orientable modulo $k$ with free generators $x_1, \dots, x_n$.
        The \emph{$\catM$-graded permutation twisted cohomology ring} is defined as follows: \[\tH_{\catM}(G;k) \coloneqq  \bigoplus_{s\in\bbZ}\bigoplus_{x \in \catM}\Hom_{\catK(G;k)}(\bbone, x[s]).\]
        More generally, let $y \in \catK(G;k)$ be an object, The \emph{$\catM$-graded permutation twisted cohomology} of $y$ is the following left $\tH_\catM(G;k)$-module: \[\tH_{\catM}(G;y) \coloneqq \bigoplus_{s\in\bbZ}\bigoplus_{x \in \catM}\Hom_{\catK(G;k)}(\bbone, x[s] \otimes y).\]
    \end{definition}

    \begin{remark}
        If $\mathsf{perm}(G;k)^\natural$ is semisimple abelian (e.g., $k$ is a field with $|G|$ invertible in $k$), then it is routine to check that the only such $\catM$-graded twisted cohomology ring is isomorphic to $k$.
    \end{remark}

    \begin{proposition}
        By \Cref{rec:comp_map_general} we have a continuous \emph{twisted comparison map}
        \begin{align*}
            \comp_\catM\colon \Spc(\catK(G;k)) &\to \Spec^h(\tH_\catM(G;k))\\
            \catP & \mapsto \langle f\in \tH_\catM(G;k) \text{ homog.}\mid \cone(f)\not\in\catP \rangle.
        \end{align*}
        The comparison map is characterized by the fact that \[\comp_\catM\inv(Z(f)) = \supp(\cone(f)) = \{\catP\in \Spc(\catK) \mid f \text{ is not invertible in } \catK(G;k)/\catP\}.\]
    \end{proposition}

    \begin{remark}\label{def:ample}
        Following \cite{Gal25}, we say $\catM$ is \emph{ample} if the opens \[
        \open(f) \coloneqq  \supp(\cone(f))^c
        \] associated to the homogeneous $f \in \tH_\catM(G;k)$ give an open basis of $\Spc(\catK(G;k))$. This generalizes the definition in algebraic geometry via Dell'Ambrogio--Stevenson's ampleness result \cite{DS14}; if a quasi-compact quasi-separated scheme $X$ has an ample family of line bundles, then one can construct a twisted cohomology ring for $X$ recovering the Balmer spectrum of $\on{D}_{\on{perf}}(X)$. We say the generators of $\catM$ form an \emph{ample family of line bundles} in this case. By \cite[Proposition 5.8]{Gal25}, this is equivalent to the twisted comparison map being an open immersion. Ultimately, our goal is to find an ample family which doesn't contain too much extraneous data, and which gives a `nice' twisted cohomology theory with respect to functoriality and local geometry.
    \end{remark}

    \subsection{Opens and central localization} \label{sec:central_loc_general} The forerunner construction \Cref{cons:forerunners} gives us an open cover the spectrum and show that the central localizations associated to each are compatible with the twisted comparison map.

    \begin{construction}
        We define an open of $\Spc(\catK(G;k))$ as follows. For $H \in \Sub_\Phi(G;k)$, let \[U_\catM(H) \coloneqq  \bigcap_{i = 1}^n \open(\iota^H_{x_i}),\] where $\open(\iota^H_{x_i}) = \supp(\iota^H_{x_i})^c$, the open locus where $\iota^H_{x_i}$ is invertible.
    \end{construction}

    \begin{theorem}
        \label{thm-opencover-Mtwistedcohomology}
        If $k$ is Noetherian, the collection of quasi-compact opens \[\{U_\catM(H)\}_{H \in \Sub_\Phi(G;k)/G}\] is an open cover of $\Spc(\catK(G;k))$.
    \end{theorem}

    \begin{proof}
        Indeed, by \Cref{prop:forerunner_props}, every point $\catM(\frakp;H)$ of $\Spc(\catK(G;k))$ belongs to $U_\catM(H)$. On the other hand, recall that any point in $\Spc(\catK(G;k(\frakp)))$ specializes to a point of the form  $\catP(H,0)$. By \Cref{rec-setbijection-spc},  and the fact that continuous maps preserve specializations, we deduce that any point in $\Spc(\catK(G;k))$ specializes to one of the form $\catM(\frakp;H)$.
    \end{proof}

    \begin{remark}
       As observed in the proof of \Cref{thm-opencover-Mtwistedcohomology}, we get that the closed points of $\Spc(\catK(G;k))$ have the form $\catM(\frakp;H)$ for some $\frakp$, and some $H\leq G$. However, not every $\catM(\frakp;H)$ is necessarily a closed point, as $\frakp \in \Spec(R)$ need not be.
    \end{remark}

    This open cover allows us to deduce the effective orientable endotrivials are tt-line bundles. Therefore, upon localization to each open, the twisted cohomology ring identifies with the usual cohomology ring, which gives us an open cover of the Balmer spectrum, hinting at our Dirac scheme structure. However, if the associated comparison map is not injective, then our collection of line bundles cannot be ample, and our cover non-affine. This next result is a generalization of \cite[Theorem 4.11]{Mil25c}.

    \begin{theorem}\label{thm:quasieffects_linebundles}
        If $k$ is Noetherian, every orientable quasieffective endotrivial $x \in \Pic(\catK(G;k))$ is a tt-line bundle.
    \end{theorem}
    \begin{proof}
        Indeed, it suffices to show this fact for effective orientable endotrivials. Let $x \in \Pic^{\on{eo}}(\catK(G;k))$ and set $\catM = \langle x\rangle$. Then in each open $U_\catM(H)$ for $H \in \Sub_\Phi(G;k)$, $\iota^H_x\colon \bbone \to x[-h_x(H)]$ (where $h_x$ is the extended h-mark homomorphism, see \Cref{rmk:extended_h_mks}) is, essentially by construction, an isomorphism in $\catK(G;k)|_{U_\catM(H)}$, and the result follows.
    \end{proof}

    \begin{definition}
        Let $\calM \leq \Pic^{\on{eo}}(\catK(G;k))$ be as in \Cref{def-Mgraded-ptcr}, let  $H \in \Sub_\Phi(G;k)$ and let $S_H \subseteq \tH_\catM(G;k)$ be the multiplicative subset generated by all $\iota^H_{x_i}$ for $i \in \{1, \dots, n\}$.
        \begin{enumerate}
            \item We define the $\bbZ$-graded ring \[\calO^\sbull_\catM(H) \coloneqq  \big(\tH_\catM(G;k)[S_H\inv]\big)_{0-\text{twist}}\] as the twist-zero part of the localization of $\tH_\catM(G;k)$ with respect to $S_H$.
            \item We define the \emph{central localization} \[\catL_\catM(H) \coloneqq  \catK(G;k)[S_H\inv]\] with respect to the central multiplicative subset $S_H$. This category has $\bbZ$-graded endomorphism ring $\calO^\sbull_\catM(H)$; see \cite{DS14, Gal25} for details.
        \end{enumerate}
    \end{definition}

    Every central localization upon idempotent completion identifies with the familiar localization with respect to a quasi-compact open.

    \begin{proposition}
        The localization with respect to the open $U_\catM(H)$, \[\catK(G;k)|_{U_\catM(H)} = (\catK(G;k)/\catK(G;k)_{U_\catM(H)^c})^\natural\] is tt-equivalent to the idempotent completion $\catL_\catM(H)^\natural$. Here, $\catK(G;k)_{U_\catM(H)^c}$ is the thick subcategory of $\catK(G;k)$ consisting of objects with support in $U_\catM(H)^c$. In particular, $\Spc(\catL_\catM(H)) \cong U_\catM(H)$.
    \end{proposition}
    \begin{proof}
        Indeed, by definition $U_\catM(H)$ has complement \[Z = \bigcup_{i=1}^n\supp(\iota_{x_i}^H)\] and the associated tt-ideal $\catK_Z$ is precisely $\langle\cone(f) \mid f \in S_H\rangle$. The identification follows by \cite[Theorem 3.6]{DS14} which gives that \[\catK(G;k)[S_H\inv]\cong \catK(G;k)/\langle\cone(f) \mid f \in S_H\rangle.\] Since $\catK(G;k)|_{U_\catM(H)}$ has spectrum $U_\catM(H)$ and idempotent completion does not change the spectrum by \cite{Bal05}, the result follows.
    \end{proof}

    It is a straightforward exercise to verify that this localization plays nice with respect to the comparison map; in particular, we have a $U_\catM(H)$-local comparison map, which is precisely Balmer's original comparison map \cite{Bal10}.

    \begin{proposition}
        For every $H \in \Sub_\Phi(G;k)$, the following square commutes.
        \begin{center}
            \begin{tikzcd}
                U_\catM(H) = \Spc(\catL_\catM(H)) \ar[r, hookrightarrow] \ar[d, "\comp_{\catL_\catM(H)}"'] & \Spc(\catK(G;k)) \ar[d, "\comp_\catM"]\\
                \Spec^h(\calO^\sbull_\catM(H))  \ar[r, hookrightarrow] & \Spec^h(\tH_\catM(G))
            \end{tikzcd}
        \end{center}
    \end{proposition}

    \subsection{Finite generation} With our general treatment of twisted cohomology rings in \Cref{sec:twisted_coh_app}, we deduce that in fact, any generalized permutation twisted cohomology ring is finitely generated, hence Noetherian. Moreover, the permutation twisted cohomology of any object in $\catK(G;k)$ is finitely generated over $\tH_\catM(G;k)$. This follows from the rather remarkable fact that $\tH_\catM(G;k)$ is finite over its subalgebra generated by the forerunners of the generators of $\catM$.

    \begin{theorem}\label{thm:perm_noeth_thm}
        Let $G$ be a finite group and $k$ be a commutative Noetherian ring. The $\catM$-graded permutation twisted cohomology ring $\tH_\catM(G;k)$ is finite over the $k$-subalgebra $R_\calF\subseteq \tH_\catM(G;k)$ generated by $\calF$, with \[\calF =  \{\iota^H_{x_i}\mid i \in \{1,\dots, n\}, H\in \Sub_{\Phi}(G)\}.\]
        In particular, $\tH_\catM(G;k)$ is a finitely generated $k$-algebra, hence Noetherian.
        Moreover, given $y \in \catK(G;k)$, its twisted cohomology $\tH_\catM(G;y)$ is a finitely generated $R_\calF$-module, hence a finitely generated $\tH_\catM(G;k)$-module.
    \end{theorem}
    \begin{proof}
        This follows from the more general \Cref{thm:general_th_thm}; we check the hypotheses are satisfied by $\tH_\catM(G;k)$. We also assume for the proof that $\id_\bbone \in \calF$ for compatibility, this does not affect $R_\calF$. First, one has $\End^\sbull_{\catK(G;k)}(\bbone) \cong k$, and for every $x \in \catK(G;k)$, \[\rk_k\bigoplus_{s\in \bbZ}\Hom_{\catK(G;k)}(\bbone, x[s]) < \infty\] so finite generation of the usual cohomology ring is satisfied. Next, by choice of $\calF$, we have $\calF_{x_i} \neq 0$ for every $i$, as $\calF_{x_i}$ is the set of all forerunners for $x_i$ at all singular subgroups of $G$.

        Now, consider the tensor product of cones of forerunners, \[c \coloneqq  \bigotimes_{H \in \Sub_\Phi(G;k)}\cone(\iota_{x_i}^{H}).\] By conservativity, it suffices to show $\check\Psi^H\circ\lambda_\frakp(c) = 0$ for all $\frakp \in \Spec(k)$ and $H \in \Sub_{\on{char}(k(\frakp)}(G)$. However, this is guaranteed, as $\cone(\iota_{x_i}^H)$ is a tensor summand of $c$, and $\check\Psi^H\circ\lambda_\frakp(\iota_{x_i}^H)$ is an isomorphism by \Cref{prop:forerunner_props}, hence $\check\Psi^H\circ\lambda_\frakp(\cone(\iota_{x_i}^H)) = 0$, as desired. The conditions of \Cref{thm:general_th_thm} are met, and the rest are straightforward consequences.
    \end{proof}

    \begin{remark}
        In fact, one need not use forerunners as explicitly constructed in \cite{GM26} or \cite{Mil25c}. One only needs morphisms whose image under modular fixed points become quasi-isomorphisms for every singular subgroup of $G$.
    \end{remark}

    \subsection{History of permutation twisted cohomology}

    We conclude by reviewing the previous permutation twisted cohomology theories considered by Balmer--Gallauer and Miller. For this section, assume $k$ is a field of prime characteristic $p$.

    \begin{definition} We describe how the previous versions of twisted cohomology introduced by Balmer--Gallauer \cite{BG25} and Miller \cite{Mil25c} fall into the generalized framework.
        \begin{enumerate}
            \item Let $\calE_p \subseteq \Pic^{\on{eo}}(\catK(G;k))$ denote the submonoid generated by the effective endotrivials \[u_N \coloneqq  \begin{cases} k(G/N) \to k & p=2 \\ k(G/N) \to k(G/N) \to k & p > 2 \end{cases}\] for $N \triangleleft G$ a normal subgroup of index $p$ and $k$ in homological degree 0. Explicitly, $u_N$ is a truncated periodic projective resolution of the trivial $k(G/N)$-module. One may observe that any $u_N$ is orientable and effective, associated to a rotation action on a 2-sphere if $p$ is odd and a reflection of a 1-sphere if $p = 2$.
            Then the $\calE_p$-graded twisted cohomology ring is precisely Balmer--Gallauer's twisted cohomology ring \cite{BG25}.
            \item If $G$ is a $p$-group, let $\Pic^+ \leq \Pic^{\on{eo}}(\catK(G;k))$ denote the submonoid generated by the images of all $V \in \Irr_\bbR(G) \setminus \{\bbR\}$ under $\catB$. The $\Pic^+$-graded twisted cohomology ring is precisely Miller's `remixed' twisted cohomology ring \cite{Mil25c}. If $G$ is elementary abelian, then the two rings coincide. This ring is also the representation-graded Bredon homology ring of $G$ with coefficients in $k$, see \cite{HS25} for details.
        \end{enumerate}
    \end{definition}

    Balmer--Gallauer deduced that when $G$ is elementary abelian, $\tH_{\calE_p}(G)$ is generated by forerunners and nilpotent morphisms, all in minimally-graded degrees.

    \begin{theorem}\cite[Theorem 12.21]{BG25}
       Let $G$ be an elementary abelian $p$-group. The $\calE_p(G)$-graded twisted cohomology ring $\tH_{\calE_p}(G)$ is a finitely generated as $k$-algebra by the global sections $\bbone \to u_N[i]$ with $i \in \{-2, -1, 0\}$ for all normal subgroups $N$ of index $p$. In particular, $\tH_{\calE_p}(G)$ is Noetherian.
    \end{theorem}

    However, this need not occur in general; for instance, $\tH_{\Pic^+}(C_4\times C_2; k)$ for $k$ a field of characteristic 2 is not generated by its minimally graded morphisms. See \cite[Example 7.18]{Mil25c}.

    \begin{theorem}\cite[Theorem 15.3, Corollary 15.6]{BG25}
        Let $G$ be an elementary abelian $p$-group, $k$ a field of characteristic $p$, and let $H \leq G$ be a subgroup. Then the comparison map restricts to a homeomorphism \[\comp_{\catL_{\calE_p}(H)}\colon U_{\calE_p}(H) \xrightarrow{\sim} \Spec^h(\calO_G^\sbull(H)).\] Consequently, the comparison map $\comp$ is an open immersion with image \[\{\frakp \in \Spec^h(\tH_{\calE_p}) \mid \text{ For all $N \in \calN$, either $\iota_{u_N}^1 \not\in \frakp$ or $\iota_{u_N}^G \not\in \frakp$}\},\] where $\calN$ denotes the set of normal subgroups of $G$ with index $p$.
    \end{theorem}

    \begin{remark}
        The $\calE_p$-graded twisted cohomology ring is somewhat tractable. For elementary abelian $p$-groups, Balmer--Gallauer gave an algebra $\underline{\calO}^\sbull_{\calE_p}(H)$ by generators and relations for which $\underline{\calO}^\sbull_{\calE_p}(H)_{\on{red}} \cong \calO^\sbull_{\calE_p}(H)_{\on{red}}$, see \cite[Theorem 17.13]{BG25}. Hausmann--Schwede deduced that when $p =2$, the two rings are in fact isomorphic, confirming Balmer--Gallauer's expectation, see \cite[Remark 3.3]{HS25}.
    \end{remark}

    For $\Pic^+$-graded twisted cohomology, Miller was able to show the analogous results for $p$-groups with a Noetherianity assumption. He deduced Noetherianity for Dedekind $p$-groups and $D_8$, but left it as a conjecture in generality.

    \begin{theorem}\cite{Mil25c}\label{thm-inj-Mgraded}
        Let $G$ be a $p$-group and $k$ a field of characteristic $p$. The comparison map \[\comp_{\Pic^+}\colon \Spc(\catK(G;k))\to \Spec^h(\tH_{\Pic^+}(G))\] is injective. For $H \leq G$, upon restriction to $U_{\Pic^+}(H)$, the comparison map \[\comp_{\catL_{\Pic^+}(H)}\colon\Spc(\catL_{\Pic^+}(H)) = U_{\Pic^+}(H) \hookrightarrow \Spec^h(\calO^\sbull_{\Pic^+}(H))\] is a homeomorphism if $\calO_{\Pic^+}^\sbull(H)$ is coherent (for instance, if $\tH_{\Pic^+}(G)$ is Noetherian). If this holds for all subgroups, $\comp_{\Pic^+}$ is an open immersion.
    \end{theorem}

    In light of \Cref{thm:perm_noeth_thm}, which in particular, answers Miller's conjecture, we immediately obtain the following.

    \begin{corollary}
        Let $G$ be a $p$-group and $k$ a field of characteristic $p$. The comparison map \[\comp_{\Pic^+}\colon \Spc(\catK(G;k))\to \Spec^h(\tH_{\Pic^+}(G))\] is an open immersion. For $H \leq G$, upon restriction to $U_{\Pic^+}(H)$, the comparison map \[\comp_{\catL_{\Pic^+}(H)}\colon U_{\Pic^+}(H) \to \Spec^h(\calO^\sbull_{\Pic^+}(H))\] is a homeomorphism.
    \end{corollary}

    \begin{remark}
        When $G$ is a $p$-group and $k$ a field of characteristic $p$, the open cover $\{U_{\Pic^+}(H)\}_{H \in \Sub_p(G)/G}$ witnesses that all quasieffective endotrivials are tt-line bundles; this cover was utilized to show this fact \cite[Theorem 4.11]{Mil25c}. It was incorrectly claimed in a previous version of \cite{Mil25c} that for remixed twisted cohomology, $\catM(H) \in U_{\Pic^+}(K)$ if and only if $H$ and $K$ are $G$-conjugate. While the `if' direction is true, \Cref{thm:non_stratify_closeds_ex} will demonstrate this to not always be the case.
    \end{remark}

    \section{Reduced twisted cohomology}

    We now introduce the `reduced' permutation twisted cohomology ring for any commutative ring. This is a more tractable version of the two previous incarnations of the permutation twisted cohomology ring, yet we will show it still captures all the geometry of the derived category of permutation modules with minimal extraneous data. The key point here is that to capture the geometry of $\catK(G;k)$, it seems that one does not need to twist by all orientable effective endotrivials, but rather, well-chosen ones which distinguish the different singular subgroups.

    We choose to twist only by the endotrivial $x\reg$ of \Cref{def:reg_orientable_endotriv}. Though we stress that this is not the only endotrivial for which one can obtain the results in the sequel (another candidate for instance could be the endotrivial obtained by tensoring all the irreducible endotrivials), this is a rather convenient one for a variety of reasons.

    \begin{definition}\label{def-red-tcring}
        Let $G$ be a finite group and $k$ a commutative ring. We define the bigraded \emph{reduced permutation twisted cohomology ring} $\tH(G)$ as follows. \[\tH(G) \coloneqq  \tH(G; k) = \tH_{\langle x\reg\rangle}(G;k)= \bigoplus_{s \in \bbZ, n \in \bbN} \Hom_{\catK(G;k)}(\bbone, x\reg^{\otimes n}[s]),\] where $\catM = \langle x\reg \rangle$ is the free monoid generated by $x\reg$. If $k$ is Noetherian, then by \Cref{thm:perm_noeth_thm}, $\tH(G)$ is finitely generated, in particular Noetherian.
    \end{definition}

    The reduced twisted cohomology ring is graded-commutative if we only consider the shift, and commutative on the nose on if we only consider the `twist'.
    By the framework of the previous section, we have a well-defined, continuous comparison map.

    \begin{proposition}\label{prop-comparison-reducedcohomology}
        Let $G$ be a finite group and $k$ be a commutative ring. We have a continuous comparison map \[\comp_G\colon \Spc(\catK(G;k)) \to \Spec^h(\tH(G))\] defined by \[\catP \in \Spc(\catK(G;k)) \mapsto \langle f\in \tH(G) \text{ \emph{homog.}}\mid \cone(f) \not\in \catP\rangle,\] satisfying \[\comp_G\inv(Z(f)) = \supp(\cone(f)) = \{\catP \in \Spc(\catK) \mid f\text{ is not invertible in }\catK(G;k)/\catP\}.\]
    \end{proposition}

    \begin{notation}\label{not-forerunner-xreg}
        Because $x\reg$ is orientable by construction, we have forerunner morphisms, see \Cref{cons:forerunners}. We write $\iota(H) \coloneqq  \iota^H_{x\reg}$ for any $H \in \Sub_\Phi(G;k)$. Note the forerunners are compatible with base change; see \Cref{rem-basechange-forerunners}. By \Cref{thm:perm_noeth_thm}, if $k$ is Noetherian, then $\tH(G)$ is finite over the subalgebra generated by the forerunners $\iota(H)$ with $H \in \Sub_\Phi(G;k).$
    \end{notation}

    \subsection{tt-functors} In addition to compatibility with central localizations, the comparison map is compatible with tt-functors to an extent due to the choice of $x\reg$ as our twist.

    \begin{observation}
        Let $F\colon \catK(G;k) \to \catK(G';k')$ be a tt-functor sending the regular endotrivial $x\reg(G)$ to $x\reg(G')^{\otimes d}[s]$ for some $d \in \bbN_+$ and $s \in \bbZ$. Then $F$ induces a ring homomorphism
        \[F\colon \tH(G;k) \to \tH(G';k')\]
        that is homogeneous; indeed, the image of a homogeneous element of degree $(n,i)$ is homogeneous of degree $(d\cdot n, n\cdot s+i)$, where $n$ corresponds to the part graded by $x\reg$ and $s$ to the shift.
        Therefore, we obtain an induced spectral map on twisted cohomology
        \[\Spec^h(F)\colon \Spec^h(\tH(G';k')) \to \Spec^h(\tH(G;k)),\] compatible with the comparison map in the following sense.
    \end{observation}

    \begin{proposition}\label{prop:comm_square_from_functor}
        Let $F\colon \catK(G;k) \to \catK(G';k')$ be a tt-functor which sends $x\reg(G)$ to $x\reg(G')^{\otimes n}$ for some $n \in \bbN$. Then the following square commutes.
        \begin{center}
            \begin{tikzcd}
                \Spc(\catK(G';k')) \ar[r, "\Spc(F)"] \ar[d, "\comp_{G'}"']& \Spc(\catK(G;k)) \ar[d, "\comp_{G}"]\\
                \Spec^h(\tH(G';k')) \ar[r, "\Spec^h(F)"]& \Spec^h(\tH(G;k))
            \end{tikzcd}
        \end{center}
    \end{proposition}
    \begin{proof}
        Routine verification using $\cone(F(f)) \cong F(\cone(f))$.
    \end{proof}

    \begin{remark}\label{rem-ringmap-ttfunctors}
        Restriction and modular fixed points (over a ring of characteristic $p$, see \Cref{subsec-generalized-modular}) both induce spectral maps on twisted cohomology. Indeed, the result can be checked from the formulas
        \[
        \Res^G_H (\bbR G) = (\bbR H)^{\oplus [G:H]} \quad \mbox{and} \quad (\bbR G)^H  = \bbR(\Weyl{G}{H})^{\oplus [G:N_G(H)]}
        \]
         via $\catB\colon \on{RO}(G) \to \Pic(\catK(G;k))$. In this sense, $\tH(-;k)$ can be viewed as a partial Green biset functor with the operations of restriction and deflation; see \cite{Bou10}.

         Similarly, base change along a ring map $k\to k'$ induces a ring homomorphism
         \[
         k'\otimes_k -\colon \tH(G;k)\to \tH(G;k').
         \]
         In particular, the fiber functors $\lambda_\frakp$ all induce maps on spectra. However, inflation does \emph{not} induce a spectral map on reduced twisted cohomology, as the $\Inf^G_{G/N}(x\reg(G/N)) \neq x\reg(G)$.
    \end{remark}

    In fact for modular fixed points, the induced map on spectra is injective for normal subgroups.

    \begin{proposition}\label{prop:mod_fix_pts_surj}
        Let $k$ be a commutative ring of characteristic $p$, and $H \in \Sub_p(G)$ be a normal subgroup. Then the induced map \[\Psi^H\colon \tH(G) \to \tH(G/H)\] is well-defined and surjective. Consequently, the induced map on spectra \[\psi^H\colon \Spec^h(\tH(G/H)) \to \Spec^h(\tH(G)))\] is injective.
    \end{proposition}
    \begin{proof}
        As observed in \Cref{rem-ringmap-ttfunctors}, we have that $(\bbR G)^H = \bbR (G/H)$, and consequently by $\catB$ (using biset functoriality), $\Psi^H(x\reg(G)) \cong x\reg(G/H)$. Thus,  well-definedness follows. For surjectivity, we have an inclusion \[i_H\colon \Inf^G_{G/H}(x\reg(G/H)^{\otimes n}) \to x\reg(G/H)^{\otimes n},\] induced by including the subcomplex of $H$-fixed points into a cellularization of the representation sphere $(S^{\bbR G})^{\wedge 2n}$. Then it is easy to see $\Psi^H(i_H) = \id$. Hence for any homogeneous $f \in \tH(G/H;k)$, i.e. a morphism $f \colon \bbone \to x\reg(G/H)^{\otimes n}$, the corresponding morphism $i_H \colon \Inf^G_{G/H}(f) \in \tH(G;k)$ demonstrates surjectivity.
    \end{proof}

    \subsection{Opens and localizations} We translate the results of \Cref{sec:central_loc_general} to the reduced twisted cohomology setting.

    \begin{definition}
        Let $G$ be a finite group and $k$ a commutative ring. Set \[U(H) = U_G(H)\coloneqq  \open(\iota(H)) = \supp(\cone(\iota(H)))^c\subseteq \Spc(\catK(G;k)),\] the open locus where $\iota(H)$ is invertible, where $\iota(H)$ is as in \Cref{not-forerunner-xreg}.
    \end{definition}

    \begin{theorem}\label{thm-opencover-reduced-ptc}
        Let $k$ be a commutative Noetherian ring, $H, K \in \Sub_\Phi(G)$, and $\frakp \in \Spec(k)$ satisfy that $K$ is a $\on{char}(k(\frakp))$-group. The `closed' point $\catM(\frakp; K)$ belongs to $U(H)$ if and only if $K=_G H$ are conjugate. In particular, $\{U(H)\}_{H \in \Sub_p(G)/G}$ is a quasi-compact open cover of $\Spc(\catK(G;k))$.
    \end{theorem}
    \begin{proof}
        One may determine from \Cref{prop:forerunner_props} that $\Psi^K\circ \lambda_\frakp(\iota(H))$ is a quasi isomorphism if and only if $H =_G K$. Since $\catM(\frakp; K) = \ker(\check\Psi^K\circ \lambda_\frakp)$ and $\cone(\iota(H)) \in \ker(\check\Psi^K\circ \lambda_\frakp)$ if and only if $\check\Psi^K\circ\lambda_\frakp(\iota(H))$ is a quasi-isomorphism, the first statement follows. The second claim is a particular case of \Cref{thm-opencover-Mtwistedcohomology}.
    \end{proof}

    In particular, if $k$ is local with maximal ideal $\frakm$, the open cover completely distinguishes the closed points of the spectrum, since the closed points of $\Spc(\catK(G;k))$ are precisely $\catM(\frakm; H)$ for $H \in \Sub_\Phi(G)$.

    \begin{corollary}\label{cor:opens_strat_closedpts}
        Let $k$ be a field of prime characteristic $p$, and $K,H \in \Sub_p(G)$. The closed point $\catM(K)$ belongs to $U(H)$ if and only if $K =_G H$ are conjugate.
    \end{corollary}

    The next few results and constructions follow directly from \Cref{sec:central_loc_general}.

    \begin{definition}\label{def:localizations}
        Let $H \in \Sub_\Phi(G;k)$, and let $S_{H} \subset \tH(G)= \tH(G;k)$ be the multiplicative subset generated by $\iota(H)$.

        \begin{enumerate}
            \item We define the $\bbZ$-graded ring
            \[\calO^\sbull(H) = \calO^\sbull_{G}(H) \coloneqq  \big(\tH(G)[S_{H}\inv]\big)_{\on{0-twist}}\] as the twist-zero part of the localization of $\tH(G)$ with respect to $S_{H}$.
            \item We define the central localization \[\catL(H) = \catL_G(H) \coloneqq  \catK(G;k)[S_{H}\inv]\] with respect to the central multiplicative subset $S_{H}$. This category has $\bbZ$-graded endomorphism ring $\calO^\sbull_{G}(H)$.
        \end{enumerate}

    \end{definition}

    \begin{proposition}\label{prop:central_locs_are_open_locs}
        Let $k$ be a commutative ring. Then the localization
        \[
        \catK(G;k)|_{U(H)} = (\catK(G;k)/\catK(G;k)_{U_H^c})^\natural
        \]
        is tt-equivalent to the idempotent-completion $(\catL(H))^\natural$. In particular, $\Spc(\catL(H)) \cong U(H)$.
    \end{proposition}

    \begin{proposition}\label{prop:comm_square_from_localization}
        For every $H \in \Sub_\Phi(G)$, the following square commutes.
        \begin{center}
            \begin{tikzcd}
                U(H) = \Spc(\catL(H)) \ar[r, hookrightarrow] \ar[d, "\comp_{\catL(H)}"'] & \Spc(\catK(G;k)) \ar[d, "\comp_G"]\\
                \Spec^h(\calO^\sbull_{G}(H))  \ar[r, hookrightarrow] & \Spec^h(\tH(G))
            \end{tikzcd}
        \end{center}
    \end{proposition}

    We now turn to the open $U(1)$ and corresponding localizations. Recall that $\on{D}_{\perm}(G;k)$ denotes the full tt-subcategory of $\on{D}_b(kG)$ thickly generated by (images of) permutation modules, see \Cref{rec-spc-dperm}.

    \begin{theorem}\label{thm:coh_open}
        Assume $k$ is a commutative Noetherian ring. The open $U(1) \subseteq \Spc(\catK(G;k))$ is equivalently the cohomological open $V_G = \supp(\catK_{\on{ac}}(G;k))^c$. Consequently,
        \[\catL(1)^\natural\cong \catK(G;k)|_{U(1)} \cong \left(\catK(G;k)/\catK_{\on{ac}}(G;k)\right)^\natural \cong \on{D}_{\perm}(G;k). \]
        Moreover, $\calO_{G}^\sbull(1) \cong \on{H}^\sbull(G;k)$ as graded $k$-algebras.
    \end{theorem}
    \begin{proof}
        First, since the dimension of the regular real representation of $G$ is precisely the order of $G$, we have $h_{x\reg}(1)=2\cdot |G|\eqqcolon n$. Moreover, note that $x\reg$ in degree $n$ is a projective $kG$-module. With these observations in hand, we deduce that
        \[y \coloneqq  \cone(\iota(1))^*\]
        i, after an appropriate shift, a bounded chain complex in non-negative homological degrees satisfying $y_0 = k$ and $y_1$ is a projective $kG$-module. Now, every $x \in \catK_{\on{ac}}(G;k)$ satisfies $x \otimes y_1 = 0$, therefore \cite[Corollary 3.21]{BG25} implies that $\catK_{\on{ac}}(G;k)$ is generated by $y$. Consequently, $\supp(\catK_{\on{ac}}(G;k)) = \supp(y)$, and hence $U(1) \coloneqq  \supp(y)^c = \supp(\catK_{\on{ac}}(G;k))^c$, as desired. The last statement follows by general tt-geometry and \Cref{prop:central_locs_are_open_locs}, as
        \[
        \on{D}_{\perm}(G;k) \cong (\catK(G;k)/\catK_{\on{ac}}(G;k))^\natural
        \]
        by \cite{BG23a}; see \Cref{rec-localization-Upsilon}.
        The second claim is a consequence of \cite[Theorem 1.3]{Lau23}  and \Cref{rec-spc-dperm}.
    \end{proof}


    \section{Immersion for reduced twisted cohomology over Noetherian bases}

    The main goal of this section is to prove that, for a finite group $G$ and Noetherian $k$, the comparison map $\comp_G$ from \Cref{prop-comparison-reducedcohomology} is an open immersion, which restricts to homeomorphisms along our open cover from \Cref{thm-opencover-reduced-ptc}. To do this, we first prove injectivity for rings of prime characteristic, then lift injectivity to commutative Noetherian rings, and handle immersion all together. We extend this to arbitrary commutative rings in the sequel.

    \subsection{Noetherian rings of prime characteristic} Assume that $k$ is a commutative Noetherian ring of prime characteristic $p$ dividing the order of $G$ (as these statements trivially hold if $|G|$ is invertible in $k$).

    For this first theorem, we label modular fixed point functors with subscripts to denote the ambient group for clarity. The following is an analogue of \cite[Proposition 6.1]{Mil25c}, and demonstrates that twisted cohomology completely `recovers' $p$-local cohomology at any $p$-subgroup. Because our twisted cohomology ring only contains twists by $x\reg$, which is already restriction-coherent up to direct sums, our proof is a bit easier than that in \cite{Mil25c}.

    \begin{theorem}\label{prop:lifting_lemma}
        Let $G$ be a finite group, $k$ a commutative ring of characteristic $p$ prime dividing $|G|$, and $H \in \Sub_p(G)$. Set $n = [G: N_G(H)]$ For every homogeneous $f \in \tH(N_G(H);k)$, there exists a homogeneous $\hat{f} \in \tH(G)$ such that \[\check\Psi_{N_G(H)}^H(f) = \check\Psi_G^H(\hat{f})\in \on{H}^\sbull(\Weyl{G}{H};k).\]
        Consequently, the composition \[\tH(G;k) \xrightarrow{} \tH(N_G(H);k) \xrightarrow{\Psi^H_{N_G(H)}} \tH(\Weyl{G}{H};k) \to \calO^\sbull_{H/H}(\Weyl{G}{H}) \cong \on{H}^\sbull(\Weyl{G}{H};k)\] is surjective.
    \end{theorem}
    \begin{proof}
        Let $f \colon k \to x[s]$ be the homogeneous element in $\tH(N_G(H);k)$, for $x = x\reg(N_G(H))^{\otimes i}$ for some $i\in \bbN$. First, we replace $f$ with
        \[
        f\otimes f' \colon k \to x\reg(N_G(H))^{\otimes i} \otimes x\reg(N_G(H))^{\otimes j}
        \]
        where $j \in \bbN$ satisfies that $n$ divides $i+j$; this ensures that the target is of the form $\Res^G_{N_G(H)}(x\reg(G)^{\otimes l})$ with $nl = i+j$. Setting $f' = \iota(H)^{\otimes j}$; then $\Psi^H_{N_G(H)}(f)$ remains unchanged, since $\iota(H) = 1$ in $\calO^\sbull_{\Weyl{G}{H}}(H/H) \cong \on{H}^\sbull(\Weyl{G}{H})$ by construction.

        From the forerunner construction \Cref{cons:forerunners}, we have a (partial) forerunner homomorphism \[\tilde\iota^H_{\on{reg}}\colon k \to \Res^G_{N_G(H)}(x\reg(G))[-2[G:H]] \cong x\reg(N_G(H))^{\otimes n}[-2[G:H]]\] with image contained in a direct sum of transitive permutation modules isomorphic to $k(N_G(H)/H)$ satisfying that $\check\Psi_{N_G(H)}^H(\tilde\iota^H_{x\reg})$ is an isomorphism. Note we take $\tilde\iota^H_{\on{reg}}$ to be the forerunner from \Cref{cons:forerunners} (associated to $N_G(H)$) on the nose, and not a tensor product of forerunners.

        We define \[\hat f \coloneqq  \tr^G_{N_G(H)}(f \otimes \tilde\iota^H_{\reg}) = \sum_{g\in [G/N_G(H)]} g(f\otimes \tilde\iota^H_{\reg}).\] We claim that \[\check\Psi^H_G(\hat f) = \check\Psi^H_{N_G(H)}(f) \in \calO^\sbull_{\Weyl{G}{H}}(H/H)\cong \on{H}^\sbull(\Weyl{G}{H};k)\] Indeed, we first observe \[\Psi^H_G(\hat f) = \Psi^H_{N_G(H)}(f) \otimes \Psi^H_{N_G(H)}(\tilde\iota^H_{\on{reg}}).\] This follows since the image of $\tilde\iota^H_{\on{reg}}$ has in homological degree 0 image contained inside a direct sum of permutation modules isomorphic to $k(N_G(H)/H)$. Therefore, the image of $g\cdot \iota^H_{\on{reg}}$ has image contained inside a direct sum of permutation modules isomorphic to $k({}^gN_G(H)/{}^gH)$. Consequently, in the trace sum $\tr^G_{N_G(H)}(f \otimes \tilde\iota^H_{\on{reg}})$, if $g \not\in N_G(H)$, then the image of $g(f\otimes \tilde\iota^H_{\on{reg}})$ is killed by $\Psi_G^H$, and the observation follows.
        Now, $\check\Psi^H_{N_G(H)}(\tilde\iota^H_{\on{reg}})$ is an isomorphism in $\on{D}_{\perm}(\Weyl{G}{H};k)$, and in fact under the identification $\calO^\sbull_{\Weyl{G}{H}}(H/H) \cong \on{H}^\sbull(\Weyl{G}{H})$ identifies with $\id_k$. We conclude \[\Psi^H_G(\hat f) = \Psi^H_{N_G(H)}(f) \otimes \Psi^H_{N_G(H)}(\tilde\iota^H_{\on{reg}}) = \Psi^H_{N_G(H)}(f)\] as desired.

        For the final statement, we have from \Cref{prop:mod_fix_pts_surj} that $\Psi^H_{N_G(H)}$ is surjective and the rightmost arrow is surjective by construction. That the composition is surjective now follows from the first half of this theorem.
    \end{proof}

    This result allows us next to deduce that the comparison map is injective when restricted to any `local' cohomological part $V_{\Weyl{G}{H}}$.

    \begin{theorem}\label{prop:map_from_local_coh}\label{thm-injectivity-diagonal}
        Let $H \in \Sub_p(G)$. We have a commutative square:
        \begin{center}
            \begin{tikzcd}
                V_{\Weyl{G}{H}}=\Spc(\on{D}_{\perm}(\Weyl{G}{H};k)) \ar[r, hookrightarrow, "\check\psi^H"] \ar[d, "\comp_{\on{D}_b}"] & \Spc(\catK(G;k)) \ar[d, "\comp_G"]\\
                \Spec^h(\on{H}^\sbull(\Weyl{G}{H};k)) \ar[r] & \Spec^h(\tH(G;k))
            \end{tikzcd}
        \end{center}
        where the bottom map is induced by the tt-functor $\Psi^H$ (see \Cref{subsec-generalized-modular}). Moreover, the diagonal is injective.
    \end{theorem}
    \begin{proof}
        Set $N=N_G(H)$, and $n = [G:N]$. Indeed, the above square is the composition of three commutative squares:

     \begin{adjustbox}{scale=0.9,center}
                    \begin{tikzcd}
                \Spc(\catL_{\Weyl{G}{H}}(H/H)) \ar[r, hookrightarrow] \ar[d, "\comp_{\catL(H/H)}"]& \Spc(\catK(\Weyl{G}{H})) \ar[r, hookrightarrow, "\psi^H"] \ar[d, "\comp_{\Weyl{G}{H}}"]& \Spc(\catK(N)) \ar[d, "\comp_{N}"]\ar[r, "\rho_{N}^G"] & \Spc(\catK(G;k)) \ar[d, "\comp_G"] \\
                \Spec^h(\calO^\sbull_{\Weyl{G}{H}}(H/H)) \ar[r, hookrightarrow] & \ar[r, hookrightarrow, "\psi^H"] \Spec^h(\tH(\Weyl{G}{H})) & \Spec^h(\tH(N)) \ar[r, "\rho_{N}^G"] & \Spec^h(\tH(G))
            \end{tikzcd}
     \end{adjustbox}
     where we omit the ground ring $k$ from the notation.
        The first square comes from \Cref{prop:comm_square_from_localization}, the second comes from \Cref{prop:comm_square_from_functor} applied to $\Psi^H\colon \catK(N;k) \to \catK(N/H;k)$, and the third comes from \Cref{prop:comm_square_from_functor} applied to $\Res^G_{N}\colon \catK(G;k) \to \catK(N;k)$.
        By \Cref{thm:coh_open}, the left-hand square identifies with
        \begin{center}
            \begin{tikzcd}
                V_{\Weyl{G}{H}}=\Spc(\on{D}_{\perm}(\Weyl{G}{H};k)) \ar[r, hookrightarrow] \ar[d, "\comp_{\on{D}_b}"] & \Spc(\catK(\Weyl{G}{H};k)) \ar[d, "\comp_{\Weyl{G}{H}}"]\\
                \Spec^h(\on{H}^\sbull(\Weyl{G}{H};k)) \ar[r] & \Spec^h(\tH(\Weyl{G}{H};k)),
            \end{tikzcd}
        \end{center}
        and we obtain the desired commutative diagram. Injectivity follows since the bottom-left path is injective by \Cref{prop:lifting_lemma}, and the leftmost downwards arrow is injective by \cite{Lau23}, see \Cref{rec-spc-dperm}.
    \end{proof}

    It remains to deduce that the images of each $V_{\Weyl{G}{H}}$ do not overlap under the comparison map.

    \begin{theorem}\label{thm:injectivity}
        The comparison map
        \[\on{comp}_{G}\colon  \Spc(\catK(G;k)) \to \Spec^h(\tH(G;k))\] is injective.
    \end{theorem}
    \begin{proof}
        Let $\catP = \catP(H_1,\mathfrak{p})$ and $\catQ = \catP(H_2, \mathfrak{q})$ in $\Spc(\catK(G;k))$ satisfy $\rho(\catP) = \rho(\catQ)$ in $\Spec^h(\tH(G;k))$. This implies that $\catP \in \open(f)$ if and only if $\catQ \in \open(f)$ for every homogeneous $f \in \Spec^h(\tH(G;k))$. In particular, for every $p$-subgroup $H\in \Sub_p(G)$, we have $\catP \in \open(\iota(H))$ if and only if $\catQ \in \open(\iota(H))$.

        Suppose for contradiction that $H_2 \neq_G H_1$. If $|H_1| < |H_2|$, then $\iota(H_1)$ and $\iota(H_2)$ have images (as morphisms of chain complexes) in different homological degree, and in particular, $\Psi^{H_2}(\iota(H_1)) = 0$. Otherwise if $|H_1| = |H_2|$ but $H_1 \neq_G H_2$, \Cref{prop:forerunner_props} implies both $\Psi^{H_2}(\iota(H_1)) = 0$ is zero and $\Psi^{H_1}(\iota(H_2)) = 0$.
        Now, consider the map $\check{\psi}^{H_2}\colon  V_{\Weyl{G}{H_2}} \hookrightarrow \Spc(\catK(G;k))$. In either case, we have:

        \begin{align*}
            (\check{\psi}^{H_2})\inv(\open(\iota(H_1)))&= \open(\cone(\check{\Psi}^{H_2}(\iota(H_1))))\\
            &= \open(0\colon  k  \to \check{\Psi}^{H_2}(x\reg[-2[G:H_1]]))\\
            &= \varnothing
        \end{align*}

        Therefore, $V_{\Weyl{G}{H_2}}\cap \open(\iota_i(H_1)) = \varnothing$ in $\Spc(\catK(G;k))$. On the other hand, $V_{\Weyl{G}{H_1}} \subseteq \open(\iota(H_1))$, since $\iota(H_1)$ is invertible in $\catK(G;k)/\catM(H_1)$. Therefore, $\catP \in\open(\iota(H_1)) $ but $\catQ \not\in \open(\iota(H_1))$, a contradiction. The remaining case of $|H_1| > |H_2|$ follows dually, and we conclude $H_1 =_G H_2$.

        Using that $\Spc(\catK(G;k)) = \bigsqcup_{H \in \Sub_p(G)/G} V_{\Weyl{G}{H}}$, we are left to verify that
        \[
        V_{\Weyl{G}{H}} \xrightarrow{\check{\psi}^{H}} \Spc(\catK(G;k)) \xrightarrow{\comp_G} \Spec^h(\tH(G;k))
        \]
        is injective. However, this map is the diagonal of the corresponding square from \Cref{prop:map_from_local_coh} which is injective.
    \end{proof}

    \begin{remark}\label{rmk:fin_gen_subring}
        We note that one may replace $\tH(G;k)$ with a finitely generated subalgebra $\on{R}^{\sbull\sbull}(G;k)$ containing the forerunners $\iota(H)$ and for which the compositions for each $H \in \Sub_p(G;k)$ \[\on{R}^{\sbull\sbull}(G;k) \hookrightarrow \tH(G;k) \twoheadrightarrow \on{H}^\sbull(\Weyl{G}{H};k)\] remain power-surjective. Then the proof of \Cref{thm:injectivity} goes through without much additional work. Since the twisted cohomology ring is rather large, we think it is an interesting question to determine a minimal subalgebra for which \Cref{thm:injectivity} still holds. We expect that in general, one requires more than the subalgebra $R_\calF$ generated only by the forerunners, but note that in small examples, $R_\calF$ is satisfactory.
    \end{remark}

 \subsection{The general Noetherian case} We now let $k$ be a commutative Noetherian ring.

\begin{proposition}
    \label{prop-base-change-tH}
    Let $G$ be a finite group, and $p\in \Phi(G;k)$. Then reduction modulo $p$ induces the following commutative diagram.
    \[
    \begin{tikzcd}
        \Spc(\catK(G;k/pk)) \ar[r,"\Spc( \lambda_\frakp)"] \ar[d,"\comp_G"] & \Spc(\catK(G;k)) \ar[d,"\comp_G"] \\
        \Spec^h(\tH(G;k/pk)) \ar[r] &  \Spec^h(\tH(G;k))
    \end{tikzcd}
    \]
\end{proposition}
  \begin{proof}
    This is a particular case of \Cref{prop:comm_square_from_functor} and the discussion in \Cref{rem-ringmap-ttfunctors}.
  \end{proof}

  \begin{proposition}
      \label{prop-powersurj-reductionp}
       Let $p\in \Phi(G;k)$. Then reduction modulo $p$
      \[
      \lambda_p\colon \tH(G;k)\to \tH(G;k/pk)
      \]
       is power-surjective.
  \end{proposition}

 This proposition is enough to show injectivity of the twisted comparison map. We postpone its proof to the end of this subsection, as it requires some setup.

 \begin{theorem}
 \label{thm-injectivity-generalcase}
     Let $G$ be a finite group and $k$ be a commutative Noetherian ring. Then the comparison map
     \[
    \comp_G\colon\Spc(\catK(G;k)) \to \Spec^h(\tH(G;k))
     \]
     is injective.
 \end{theorem}

\begin{proof}
     By \cite[Theorem 6.18]{Gom25}, the fiber functors induce a set partition of the spectrum:
    \[
    \bigsqcup\Spc(\lambda_\frakp) \colon \bigsqcup_{\frakp\in \Spec (k)}\Spc(\catK(G;k(\frakp))) \xrightarrow[]{\simeq} \Spc(\catK(G;k)).
    \]
    It suffices to show that each $\comp_G\circ \Spc(\lambda_\frakp)$ is an injective map in order to deduce injectivity of $\comp_G$.
   Indeed, note that the classical comparison map factors via the twisted one
    \[
    \comp\colon \Spc(\catK(G;k)) \to \Spec^h(\tH(G;k)) \to \Spec(k)
    \]
    where the right hand side map is induced by the algebra map $k\to \tH(G;k)$. Thus, naturality of this comparison map, gives us that the $\comp_G\circ \Spc(\lambda_\frakp)$ have different images for different $\frakp$.

    Let $\frakp$ be such that $\mathrm{char}(k(\frakp))$ is coprime to the order of $G$. Then $\comp_G\circ \Spc(\lambda_\frakp)$ is trivially injective since its domain is a point. For primes $\frakp$ whose residue field is a singular prime of $G$, we use reduction of coefficients. In view of \Cref{prop-reduction-modp}, it is enough to show that for each $p\in \Phi(G;k)$,
    \[
    \Spc(\catK(G;k/pk))\xrightarrow[]{\Spc(\lambda_p)} \Spc(\catK(G;k)) \xrightarrow[]{\comp_G}\Spec^h(\tH(G;k))
    \]
   is injective. Here $\lambda_p$ denotes reduction modulo $p$. By \Cref{prop-base-change-tH}, this composite agrees with
   \[
   \Spc(\catK(G;k/pk))\xrightarrow[]{\comp_{G;k}} \Spec^h(\tH(G;k/pk)) \xrightarrow[]{\Spec^h(\lambda_p)}\Spec^h(\tH(G;k)).
   \]
   The left hand side map is injective by \Cref{thm:injectivity}, and the right hand side map is injective since $\lambda_p$ is power-surjective by \Cref{prop-powersurj-reductionp}. Hence the result follows.
\end{proof}

We now turn to the proof of \Cref{prop-powersurj-reductionp}. We need some preparation.

\begin{recollection}
\label{rec-tensor-algebra}
    Let $x$ be a complex of permutation $kG$-modules. Recall that $\hom_{\catK(G;k)}(\mathbbm{1},x)\cong x^G$, and for any ring map $k\to k'$, one gets that $x^G\otimes_k k'\cong (x\otimes k')^G$. However, $\on{H}_\sbull(x^G)\otimes k'$  does not necessarily agree with $\on{H}_\sbull(x^G\otimes k')$.

  We can rewrite the graded pieces of our reduced twisted cohomology ring as
  \[
  \on{H}^{s,n}(G;k)\cong \on{H}_{-s}( (x\reg^{\otimes n})^G)
  \]
where the right hand side denotes the homology of the complex.

This lead us to consider
    \[
 A\reg^\sbull(k)=A\reg^\sbull \coloneqq \on{T}(x\reg)^G= \bigoplus_{n\in \mathbb N} (x\reg^{\otimes n})^G
\]
which is by definition the algebra of invariants of the tensor algebra $\on{T}(x\reg)$. Since the latter is a dg-algebra,  $ A\reg^\sbull$ is a dg-algebra as well. Moreover, it is graded-commutative given that the swap map $x\reg\otimes x\reg \to x\reg\otimes x\reg$ is homotopic to the identity by our choice of $x\reg$.
By the previous discussion, we get that $\tH(G;k)$ can be recovered by $\on{H}_\sbull( A\reg^\sbull)$.
\end{recollection}

The following is motivated by \cite[Proposition 54]{vdKF10} which shows that reduction modulo $p$ is power surjective on rational cohomology of certain algebraic groups.

\begin{lemma}\label{lemma-integral-reduction-modp}
    Reduction modulo $p$
    \[
      \lambda_p\colon \tH(G;\mathbb Z)\to \tH(G;\mathbb F_p)
      \]
       is power-surjective.
\end{lemma}

\begin{proof}
     In view of \Cref{rec-tensor-algebra}, our claim is equivalent to verify that reduction modulo $p$
    \[
      \on{H}_\sbull(A\reg^\sbull(\mathbb Z))\to \on{H}_\sbull(A\reg^\sbull(\mathbb F_p))
    \]
    is power surjective. To this end, note that  for each $n\in\mathbb N$, the composite
    \[
    (x\reg^{\otimes n})^G \xrightarrow[]{\eta} x\reg^{\otimes n} \xrightarrow[]{\mathrm{tr}} (x\reg^{\otimes n})^G
    \]
    is multiplication by the order of the group,  where $\mathrm{tr}$ denotes the transfer on chain complexes (aka the norm map) and $\eta$ is the inclusion of the complex of invariants. Now, recall that each $x\reg^{\otimes n}$ comes from a representation sphere, hence it has non-trivial homology in exactly one degree and it is isomorphic to $\mathbb{Z}$.  We deduce that the order of the group kills the torsion part of $\on{H}_\sbull((x\reg^{\otimes n})^G)$, hence $\on{H}_\sbull(A\reg^\sbull(\mathbb Z))$ has bounded torsion. In particular, the greatest power of $p$ dividing $|G|$, say $p^N$, kills the $p$-primary torsion of the $\on{H}_\sbull(A\reg^\sbull(\mathbb Z))$.

    Now, since all graded pieces of $A\reg^\sbull$ are torsion free, multiplication by $p$
    \[
     A\reg^\sbull(\mathbb{Z})\xrightarrow[]{p} A\reg^\sbull(\mathbb{Z}) \to A\reg^\sbull(\mathbb{F}_p)
    \]
    gives a short exact sequence that induces a multiplicative Bockstein spectral sequence (e.g., see \cite{Pal})
    \[
    E_1 = \on{H}_\sbull( A\reg^\sbull(\mathbb F_p))\implies (\on{H}_\sbull(A\reg^\sbull)/\mathrm{torsion})\otimes \mathbb F_p.
    \]
    Moreover, it also induces a change of coefficients long exact sequence on homology
    \[
    \ldots \on{H}_\sbull( A\reg^\sbull(\mathbb Z)) \xrightarrow[]{} \on{H}_\sbull( A\reg^\sbull(\mathbb F_p)) \xrightarrow[]{\beta} \on{H}_{\sbull-1}( A\reg^\sbull(\mathbb Z))\to \ldots
    \]

    Now, for $r\geq 1$, let $Z_r$ denote the subset of $E_1$ consisting of those classes surviving the $r$th page of the spectral sequence. Explicitly, $Z_r$ consist of those  classes $a\in \on{H}_\sbull( A\reg^\sbull(\mathbb F_p))$ such that $\beta(a)\in p^{r-1}\on{H}_{\sbull-1}(A\reg^\sbull(\mathbb Z))$. In particular, $Z_1=E_1$. Given that each page is graded commutative, and that each differential is a derivation, we deduce that taking $p$th-powers maps $Z_r$ to $Z_{r+1}$.

    Now, fix a class $a\in \on{H}_\sbull( A\reg^\sbull(\mathbb F_p))$ in the first page of the Bockstein spectral sequence. As discussed above, for any $r\geq 1$, $a^{p^r}$ lies in $Z_{r+1}$. On the other hand, by the long exact sequence obtained by multiplication  by $p$, we get $p\beta(a^{p^r})=0$. Taking $r=N$
    gives us that $\beta(a^{p^{N}})=0$. Indeed, $\beta(a^{p^N})$ must be of the form $p^{N}b$ for some class $b$, and hence $p^{N+1} b=0$. Then $b$ is $p$-primary torsion. Using the torsion bound,  it follows that $p^N b=0$.
    We deduce that  $a^{p^N}$ lies in the kernel of $\beta$, and using the long exact sequence for change of coefficients we obtain the desired lift.
\end{proof}

\begin{proof}[Proof of \Cref{prop-powersurj-reductionp}]
    Recall that $k$ is a commutative Noetherian ring such that $p$ is not invertible in $k$. Then $k/pk$ is an $\mathbb F_p$-algebra, and hence it is flat over $\mathbb F_p$. It follows that $\on{H}_\sbull(A\reg^\sbull(k/pk))\cong \on{H}_\sbull(A\reg^\sbull(\mathbb F_p))\otimes k/pk$. Then we can write a homogeneous element $\on{H}_\sbull(A\reg^\sbull(k/pk))$ as a linear combination of elements $a \otimes \alpha$  with $a\in \on{H}_\sbull(A\reg^\sbull(\mathbb F_p))$ and $\alpha\in k/pk$. Hence, by \Cref{lemma-integral-reduction-modp} we have that a power $a^{p^N}$ is in the image of the reduction modulo $p$, for some $N$. It follows that $a^{p^N}\otimes \alpha^{p^N}$ must lie in the image of the reduction map, as desired.
\end{proof}

    \subsection{Open immersion}
We are ready to prove that the twisted comparison map is an open immersion for Noetherian base rings.

    \begin{theorem}
    \label{thm:open_immersion_subring}
        Let $G$ be a finite group and $k$ a commutative Noetherian ring. The comparison map
        \[\comp_G\colon \Spc(\catK(G;k)) \to \Spec^h(\tH(G;k))\] is an open immersion, which restricts to a homeomorphism \[\comp_{\catL_G(H)}\colon U(H)=\Spc(\catL(H)) \to \Spec^h(\calO_{G}^\sbull(H))\]
        for all $H \in \Sub_\Phi(G)$.
    \end{theorem}

    \begin{proof}
        Since $\{U(H)\mid H\in \Sub_\Phi(G)\}$ is an open cover of $\Spc(\catK(G;k))$ and $\comp_G$ is injective by \Cref{thm:injectivity}, it suffices to prove the second claim.

        Let $H \in \Sub_\Phi(G)$.  In view of \Cref{prop:central_locs_are_open_locs}, and the fact that idempotent completion does not change the spectrum, we may instead work with $\catK(G;k)|_{U(H)}$ and its corresponding comparison map
        \[
            \comp_{U(H)}\colon U(H)= \Spc(\catK(G;k)|_{U(H)}) \to \Spec^h(\calO_{G}^\sbull(H)).
        \]
        Since $\comp_G$ is injective by \Cref{thm:injectivity}, $\comp_{U(H)}$ is injective as well.
        Moreover, since $\tH(G;k)$ is finitely generated, so is $\calO_G^\sbull(H)$. We may therefore invoke
        \cite[Theorem 10.8]{San25}
        which gives the desired result.
    \end{proof}

    \section{Geometry analysis}\label{sec-covers}
    This section is trifold. We show
    that the twisted comparison map gives a  homeomorphism to the projectivization of the homogeneous spectrum of the reduced twisted cohomology ring; we verify that the open cover coming from Miller's remixed twisted cohomology ring does not separate the closed points of the spectrum; and we analyze the geometric open of the spectrum.

    \subsection{The closed complement of the image} We determine the closed complement of $\im(\comp_G)$. It isn't much.

    \begin{definition}
    \label{rem-quotient-tH_+}
        We have (see \Cref{rmk:closed_immersion}) that $\Spec^h(\tH(G;k))$ has closed subspace $\Spec(k)$, which arises as the closed complement of the support $V(\tH_+(G;k))$ of the \emph{irrelevant ideal} \[\tH_+(G;k) \coloneqq  \bigoplus_{s\in \bbZ, n> 0}\on{H}^{s,n}(G;k).\] Indeed, this arises from the projection $\tH(G;k) \to \End_{\catK(G;k)}^\sbull(\bbone)\cong k$; the induced map on spectra sends a prime $\frakp \in \Spec(k)$ to $\frakp \oplus \tH_+(G;k) \subset \tH(G;k)$. Let \[\Proj(\tH(G;k)) = \{\frakp \in \Spec^h(\tH(G;k)) \mid \tH_+(G;k) \not\subseteq \frakp\},\] equipped with the usual graded Zariski topology.
    \end{definition}

    \begin{proposition}\label{prop:ideal_gen_by_forerunners}
        The radical ideal generated by $\calF$ is $\sqrt{\tH_+(G;k)}$, and there exists some $n \in \bbZ$ for which \[ \langle \iota(H) \mid H \in \Sub_\Phi(G;k)\rangle \supseteq \bigoplus_{s \in \bbZ, i \geq n} \on{H}^{s,i}(G;k).\]
    \end{proposition}
    \begin{proof}
        The former statement easily follows from the latter, and the latter follows from finiteness of the extension $R_\calF \subseteq \tH(G;k)$. Indeed, if $f_1, \dots, f_m$ is a list of generators of $\tH(G;k)$, viewed as a module over $R_\calF$, then setting $n$ to be one more than the maximal $x\reg$-degree of all the $f_m$ terms suffices.
    \end{proof}

    It turns out that the copy of $\Spec(k)$ is the closed complement of the image of $\Spc(\catK(G;k))$ in $\Spec^h(\tH(G;k))$ under the twisted comparison map.

    \begin{theorem}\label{thm:closed_comp_of_support}
        Let $G$ be a finite group and $k$ a commutative Noetherian ring. The open image of \[\comp_G\colon \Spc(\catK(G;k)) \to \Spec^h(\tH(G;k))\] has closed complement the `irrelevant' \[\Spec(k) \cong \Spec^h(\on{H}^\sbull(\catK(G;k))) \subseteq \Spec^h(\tH(G;k)).\]
        In particular, if $k$ is a field, then the closed complement is the unique closed point of $\Spec^h(\tH(G;k))$.
    \end{theorem}
    \begin{proof}
        In light of the \Cref{rem-quotient-tH_+}, it suffices to show $V(\tH_+(G))^c = \im(\comp_G)$. From \Cref{thm-injectivity-generalcase} and \Cref{def:localizations}, we have $\comp_G(U(H)) = V(\iota(H))^c$. Since the $U(H)$'s are an open cover, we have \[\im(\comp_G) = \bigcup_{H \in \Sub_\Phi(G;k)}  V(\iota(H))^c,\] and since the forerunners generate the radical closure of $\tH_+(G;k)$, we have \[\im(\comp_G)  = V(\tH_+(G;k))^c.\] Since $V(\tH_+(G;k)) = \Spec^h(\tH(G;k)/\tH_+(G;k)) \cong \Spec(k)$, we are done.
    \end{proof}

    \begin{corollary}\label{cor-spc=proj}
         The twisted comparison map induces a homeomorphism \[\comp_G\colon \Spc(\catK(G;k)) \xrightarrow{\sim} \Proj(\tH(G;k)).\]
    \end{corollary}
    \begin{proof}
        $\Proj(\tH(G;k))$ is precisely the open subspace
        \[
        V(\tH_+(G;k)) \subset \Spec^h(\tH(G;k))
        \]
        which by \Cref{thm:closed_comp_of_support} is precisely the image of the comparison map.
    \end{proof}

    \subsection{Open covers and tensor-compatibility of forerunners} Next, we address two open questions from \cite{Mil25c} that will, in particular, allow us to distinguish between the two open covers $\{U_{\Pic^+}(H)\}$ of \cite{Mil25c} (when $G$ is a $p$-group) and $\{U(H)\}$. We first set some notation.

    \begin{notation}
        Given two morphisms $f, g \in \Hom_{\catK(G;k)}(x,y)$, we write $f \doteq g$ if $f = ag$ for some $a \in k^\times$.
    \end{notation}

    Let $G$ be a $p$-group. In \cite{Mil25c}, Miller posed the question of whether, given any two effective endotrivials $x_1, x_2 \in \Pic(\catK(G;k))$ with associated cellularized representation spheres $S^{V_1}, S^{V_2}$ and any $H \in \Sub_p(G)$, one always has an equivalence $\iota^H_{x_1} \otimes \iota^H_{x_2} \doteq \iota^H_{x_1 \otimes x_2}.$ This holds if $H \in \{1,G\}$, see \cite[Remark 4.6(a)]{Mil25c}.

    Additionally, in \textit{loc. cit.}, Miller posed the question of whether the open cover $\{U_{\Pic^+(G)}(H)\}$ satisfies that $\catM(K) \in U_{\Pic^+}(H)$ if and only if $K =_G H$ are conjugate. He was able to show it for $p$-groups satisfying which `detect non-conjugacy:' $p$-groups satisfying that for any pair of non-conjugate subgroups $H,K \leq G$, there exists a Borel-Smith function $f$ satisfying $f(H) \neq f(K)$. If $G$ has a pair of non-conjugate subgroups of equal order for which there exists a pairing between elements of $H$ and $K$ with all paired elements conjugate, then $G$ does not detect non-conjugacy.

    We will give an example verifying that the answer to both questions is, in general, no.
    For this section, let $G = \on{UT}_3(\bbF_4)$, $p = 2$, and $k$ a field of characteristic $p$. We have $|G| = 64$. Write \[(a,b,c) \coloneqq  \begin{pmatrix}1 & a & c\\
        0 &1 & b \\
        0 & 0 & 1\end{pmatrix}\] and define $H \coloneqq  \{ (x,0,0) \mid x \in \bbF_4\}$ and $K  \coloneqq  \{(x, 0, x^2 )\mid x \in \bbF_4\}$.

    The following are straightforward computations and are left for the reader to check.

    \begin{lemma}\label{lem:ex_lems}
     The following properties hold.
        \begin{enumerate}
            \item The subgroups $H$ and $K$ are not $G$-conjugates, but are indistinguishable in the sense of \cite[Definition 4.11]{Mil25c}. That is, there exists a pairing between elements of $H$ and $K$ such that all paired elements are conjugate.
            \item We have $Z(G) = \{(0,0, x) \mid x\in \bbF_4\} \cong V_4$, and for every $1 < Z \leq Z(G)$, $HZ/Z$ and $KZ/Z$ are $G/Z$-conjugates.
        \end{enumerate}
    \end{lemma}

    \begin{proposition}\label{prop:non_faithful_irreps}
        Every irreducible real representation of $G$ is inflated from an irreducible real representation of $G/Z'$ with $1 < Z' \leq Z(G)$.
    \end{proposition}
    \begin{proof}
        Follows from the standard fact that a $p$-group $G$ has a faithful irreducible representation if and only if $Z(G)$ is cyclic.
    \end{proof}

    \begin{theorem}\label{thm:non_stratify_closeds_ex}
        For every irreducible real representation $V$ of $G$, we have $\iota^H_{x_V} \doteq \iota^K_{x_V}$. Consequently, $U_{\Pic^+}(K) = U_{\Pic^+}(H)$ and $\catM(H),\catM(K) \in U_{\Pic^+}(H) = U_{\Pic^+}(K)$.
    \end{theorem}
    \begin{proof}
        Suppose $V$ has kernel $N$. By \Cref{prop:non_faithful_irreps}, there exists a nontrivial subgroup $Z \leq Z(G)$ such that $Z \leq N$. Therefore, viewing $V$ as a faithful $G/N$-representation, we have $\iota^{HN/N}_{x_V} \doteq \iota^{KN/N}_{x_V}$. Therefore $\iota^{HN}_{x_V} \doteq \iota^{KN}_{x_V}$, viewing $V$ as a $G$-representation. Now since $N$ acts trivially on $S^V,$ we have $\dim V^H = \dim V^{NH} = \dim V^{NK} = \dim V^K$, and we conclude $\iota^H_{x_V} = \iota^K_{x_V}$ by the properties given in \Cref{prop:forerunner_props}. The final statement follows via construction of the opens $U_{\Pic^+}$.
    \end{proof}

    \begin{corollary}
        There exist forerunners $\iota^H_{x_1}$, $\iota^H_{x_2}$ for which $\iota^H_{x_1} \otimes \iota^H_{x_2} \not\doteq \iota^H_{x_1 \otimes x_2}$ in $\catK(G;k)$.
    \end{corollary}
    \begin{proof}
        We have from \Cref{prop:forerunner_props} that $\Psi^{H'}(\iota^H_{x\reg})$ is a quasi-isomorphism if and only if $H'=_G H$, but \[\Psi^K\left(\bigotimes_{V \mid \bbR G} \iota^H_{x_V}\right)\] is a quasi-isomorphism by \Cref{thm:non_stratify_closeds_ex}. Indeed, $\catM(K) \in U_\calM(H)$ if and only if $\catM(K) \in \open(\iota^H_{x_V})$ for all irreducible real representations $V$ of $G$. Since $\catM_K = \ker(\check\Psi^K)$, the result follows.
    \end{proof}

    \begin{remark}
        \Cref{thm:non_stratify_closeds_ex} asserts that in fact, it is impossible to distinguish $\catM(H), \catM(K)$ by \emph{any} `forerunner' construction satisfying the desirable conditions of \Cref{prop:forerunner_props} for the irreducible endotrivials alone; to obtain a finer open cover, one must invert forerunners for non-irreducible endotrivials. In light of this, along with \Cref{thm-opencover-reduced-ptc} which demonstrates that reduced twisted cohomology always distinguishes the closed points, it seems that our reduced construction seems to be a superior choice for a multigraded ring, although we are in no way claiming it is canonical.
    \end{remark}

    \subsection{The geometric open}\label{sec:geo_open}

    In \cite{BG25}, for elementary abelian $p$-groups $G$ over a field $k$ of characteristic $p$, the authors baptized the open $U_{\calE_p}(G)$ the \emph{geometric open}. In this case by \cite[Proposition 13.17]{BG25}, one has an identification \[\catK(G;k)|_{U_{\calE_p}}(G) \cong \frac{\catK(G;k)}{\langle k(G/H) \mid H < G\rangle},\] which is analogous to the way that geometric fixed points are constructed in topology, see e.g. \cite{MNN17}. We show the analogous result holds here (assuming $k$ a field of prime characteristic) and describe the open explicitly.

    \begin{proposition}\label{prop:geo_open_description}
        Let $G$ be a finite group and $k$ a field of prime characteristic $p$. For any Sylow $p$-subgroup $S \in \Syl_p(G)$, the open $U(S) \subseteq \Spc(\catK(G;k))$ is equivalently the points $\catP\in \Spc(\catK(G;k))$ specializing to $\catM(S)$. Consequently, $\catK(G;k)|_{U(S)} \cong \catK(G;k)/\catM(S).$
    \end{proposition}
    \begin{proof}
        In light of \Cref{thm:open_immersion_subring}, it suffices to show the open subset  $V(\iota(S))^c$ of $\Spec^h(\tH(G;k))$ is equivalently the set of primes specializing to \[\comp_G(\catM(S)) = \langle f\in \tH(G;k) \text{ homog.} \mid \check\Psi^S(f) \text{ is not an isomorphism} \rangle.\] However, this follows easily since the only homogeneous elements $ f\in \tH(G;k)$ satisfying that $\check\Psi^S(f)$ is an isomorphism are tensor products of $\iota(S)$ (up to scaling), as the lowest nonzero homological degree of an indecomposable representative of $x\reg^{\otimes n}$ is isomorphic to $k$. See \cite[Proposition 2.8]{Mil25c}, noting the $p$-group hypothesis is only used to assert the $k$-dimension 1 summand is isomorphic to the trivial $kG$-module. Thus a prime $\catP$ does not contain any homogeneous $f \in \tH(G;k)$ satisfying $\check\Psi^S(f)$ an isomorphism if and only if it does not contain $\iota(S)$, whence the result. The last statement is a general tt-geometric fact using that the generalization closure of $\catM(S)$ corresponds to $\supp(\catM(S))^c$.
    \end{proof}

    \begin{remark}
        The above characterization of the open also holds for $U(1) \cong V_G$, but need not hold for proper nontrivial $p$-subgroups in general. This can already be observed for $V_4$. See \cite[Example 16.16]{BG25}, we leave it as an exercise to the reader to compute the opens $U(H)$ for $H$ of order 2.
    \end{remark}

    \begin{corollary}\label{cor:open_identification}
        The closed complement of $U(S)$ is equivalently \[\bigcup_{H \in \Sub_p(G)\setminus \Syl_p(G)} \supp(k(G/H)).\] Consequently, for $S \in \Syl_p(G)$, we have an identification \[\catK(G;k)|_{U(S)} = \frac{\catK(G;k)}{\langle k(G/H) \mid H\in \Sub_p(G)\setminus \Syl_p(G)\rangle}.\]
    \end{corollary}
    \begin{proof}
        The closed complement of $U(S)$ is by definition $\supp(\cone(\iota(S)))$. An indecomposable representative of $\cone(\iota(S))$ contains only direct summands of permutation modules isomorphic to $k(G/H)$ with $H \in \Sub_p(G)$ non-Sylow, again by an argument analogous to \cite[Proposition 2.8]{Mil25c}. Therefore, \[\supp(\cone(\iota(S))) \subseteq \bigcup_{H \in \Sub_p(G)\setminus \Syl_p(G)} \supp(k(G/H)).\] On the other hand, for any $H < S$, $\supp(k(G/H)) \subseteq \supp(\catM(S))$, given that $\Psi^S(k(G/H)) = 0$ and $\catM(S) = \ker(\check\Psi^S)$. Since $\supp(\cone(\iota(S))) = \supp(\catM(S))$ by \Cref{prop:geo_open_description}, we conclude the two closed sets are the same. The last statement is immediate by general tt-geometry.
    \end{proof}

    \begin{remark}
        For any essentially small tt-category $\catK$ and quasi-compact open $U \subseteq \Spc(\catK)$, one can define the \emph{$U$-local Picard group} of $\catK$ as the Picard group of $\catK|_U$. Given an open cover of $\Spc(\catK)$, relating the local Picard groups $\Pic(\catK|_U)$ to $\Pic(\catK)$ is an intriguing descent-theoretic question. See e.g. \cite{Bal10c, MS16}.
    \end{remark}

    For permutation modules, since $\catK(G;k)|_{U(1)} \cong \on{D}_b(kG)$ (still assuming $k$ is a field), we have a straightforward isomorphism $\Pic(\catK(G;k)|_{U(1)}) \cong \bbZ \times \Hom(G,k^\times)$ via the K\"unneth formula. We will deduce the $U(S)$-local Picard group as well; here the theory of vertices and sources come into play, as does Lassueur's generalized theory of Dade groups for finite groups \cite{Las13}. See \cite[Section 5.1]{Li18a} for a more detailed overview of vertices and sources.

    \begin{recollection}
        We say a $kG$-module $M$ is \emph{projective with respect to $H$} ($H$-projective) for a subgroup $H$ of $G$ if $M$ is a direct summand of $\Ind^G_H N$ for some $kH$-module $N$. It is easy to see 1-projectivity is equivalent to projectivity for modules. If $M$ is indecomposable, we say $H$ is a \emph{vertex} of $M$ if $M$ is $H$-projective and $M$ is not $K$-projective for any $K < H$, and correspondingly, $N$ is a \emph{source} of $M$. It is a good exercise to show that the vertices of $M$ form a unique conjugacy class of $p$-subgroups of $G$, and an indecomposable $kG$-module is $p$-permutation if and only if it has trivial source.

        Moreover, it follows from the formula for the product of $G$-sets that if $M,N$ are trivial source modules with vertices $H$ and $K$ respectively, then every indecomposable direct summand of $M \otimes N$ has a vertex contained in $H \cap {}^gK$ for some $g \in G$. Therefore, the category $\catK_{U(S)^c}$ is equivalently the thick closure of the $k(G/H)$'s with $H \in \Sub_p(G)\setminus \Syl_p(G)$; it is automatically a thick tensor-ideal.
    \end{recollection}

    \begin{lemma}\label{lem:iso_classes_in_US}
        Let $x\in \catK(G;k)$ be indecomposable. If for some $n\in \bbZ$, $x_n$ contains an indecomposable direct summand $M$ with vertex $S$, then any chain complex $y \in \catK(G;k)$ satisfying $x \cong y \in \catK(G;k)|_{U(S)}$ also satisfies that $y_n$ admits $M$ as a direct summand.
    \end{lemma}
    \begin{proof}
     Indeed, $x \cong y \in \catK(G;k)|_{U(S)}$ if and only if there exists a roof $x \xleftarrow{g} z \xrightarrow{f} y$ for which $\cone(f), \cone(g) \in \catK_{U(S)^c}$, since $\catK_{U(S)^c}$ is thick. From the identification \[\catK_{U(S)^c} = \langle k(G/H)\mid H < S\rangle\] it follows that every complex with no contractible summands in $\catK_{U(S)^c}$ contains no indecomposable module summands with vertex $S$ in any homological degree. Therefore, $z$ must contain $M$ in the same homological degree, since it must split off in $\cone(g)$. The same holds for $\cone(f)$ and $y$, and we are done.
    \end{proof}

    \begin{proposition}\label{prop:rel_endotriv_p_perms}
        Let $M$ be a trivial source (i.e., indecomposable $p$-permutation) module with vertex $S$. Then $M \otimes M^*$ admits $k$ as a direct summand, and all other direct summands have vertex strictly contained in $S$ if and only if $\Psi^S(M)$ has $k$-dimension 1 if and only if the Green correspondent of $M$ has $k$-dimension 1.
    \end{proposition}
    \begin{proof}
        The first statement is immediate by \cite[Lemma 5.1]{Las13}. Indeed, $M$ is endo-$p$-permutation since it is already $p$-permutation. Note the lemma assumes algebraic closure of $k$, but it is not needed for our claim, it is only used for multiplicity. See also \cite[Theorem 2.1]{BC86}.

        Now, we may write $M \otimes M^* \cong k \oplus N$ for some $p$-permutation module $N$, and it is easy to see that all summands of $N$ have vertex strictly contained in $S$ if and only if $\Psi^S(N) = 0$ if and only if $\Psi^S(M) \otimes \Psi^S(M^*)$ has $k$-dimension 1 if and only if $\Psi^S(M)$ has $k$-dimension 1. The last equivalence follow from the fact that here, the Green correspondent of $M$ is $\Inf^{N_G(S)}_{N_G(S)/S}\Psi^S(M)$, see \cite[3.4]{Bro85}.
    \end{proof}

    \begin{remark}
     The explicit description (e.g., dimension or simplicity) of the direct summands of $k(G/S)$ satisfying \Cref{prop:rel_endotriv_p_perms} is in general a difficult question and related to Alperin's weight conjecture. Indeed, these modules correspond exactly to the linear $p$-weights of $G$ with $p$-subgroup $S$.
    \end{remark}

    \begin{notation}\label{not:dade}
        We say a module satisfying the latter condition of \Cref{prop:rel_endotriv_p_perms} is \emph{capped}, since such a module is precisely a capped endo-$p$-permutation module in the sense of \cite{Las13} which is already $p$-permutation. They form a subgroup of the Dade group $D_k(G)$ in the sense of \cite{Las13} which by \Cref{prop:rel_endotriv_p_perms} is isomorphic to $\Hom(\Weyl{G}{S}, k^\times)$.

        The capped modules can easily be seen to be invertible in $\catK(G;k)|_{U(S)}$. Denote the subgroup of $\Pic(\catK(G;k))|_{U(S)}$ consisting of these elements (regarded as chain complexes with support in homological degree 0) by $\on{cap}(G;k)$. One may verify given the description of $\catK(G;k)|_{U(S)}$ in \Cref{cor:open_identification} that all such modules remain non-isomorphic.
    \end{notation}

    \begin{lemma}\label{lem:source_split_lemma}
        Let $f\colon M\to N$ be a morphism between trivial source $kG$-modules with Sylow vertices. Then $f$ is an isomorphism if and only if $\Psi^S(f)$ is an isomorphism.
    \end{lemma}
    \begin{proof}
        Follows from \cite[Proposition 5.10.7]{Li18a} using rigidity.
    \end{proof}

    \begin{remark}
        Since $\ker(\check\Psi^H) = \catM(H)$ for any $H \in \Sub_p(G)$, the tt-functor \[\check\Psi^H\colon \catK(G;k) \to \on{D}_b(\Weyl{G}{H};k)\] factors as \[\catK(G;k) \to\catK(G;k)|_{U(H)} \to \catK(G;k)/\catM(H) \to  \on{D}_b(\Weyl{G}{H};k).\]
    \end{remark}

    \begin{theorem}
        Let $G$ be a finite group, $S \in \Syl_p(G)$, and $k$ a field of characteristic $p$. Then \[\Pic(\catK(G;k)|_{U(S)}) \cong \bbZ \times \Hom(\Weyl{G}{S}, k^\times).\]
    \end{theorem}
    \begin{proof}
        We claim that any bounded chain complex $x$ of $p$-permutation $kG$-modules containing no contractible direct summands which is invertible in $\catK(G;k)|_{U(S)}$ must have exactly one capped module $M$ as a direct summand amongst all its homological degrees. To show this, we check $\check\Psi^S(x)$ (noting this is well-defined for $\catK(G;k)|_{U(S)}$). We have that $\check\Psi^S(x)$ is isomorphic to a complex supported in one degree. However, using semisimplicity of $k(\Weyl{G}{S})$, since $x$ contains no contractible summands, neither does $\Psi^S(x)$ by \Cref{lem:source_split_lemma}. Therefore, $\Psi^S(x)$ is isomorphic to a module in one homological degree, hence $x$ contains exactly one indecomposable module summand not killed by $\Psi^S$, which must have Sylow vertex. By \Cref{prop:rel_endotriv_p_perms}, this summand must be capped, as we wished to show.

        Now by \Cref{lem:iso_classes_in_US} and \Cref{cor:open_identification}, it follows by a standard inductive argument that $x$ is in fact isomorphic to $M[i]$ in $\catK(G;k)|_{U(S)}$ where $M$ is a direct summand of $x$ in homological degree $i$. The result now follows from the well-defined homomorphism $x\cong M[i] \mapsto (i, M) \in \bbZ \times D_k(G)$, and the image of this map by the remarks in \Cref{not:dade} has image isomorphic to $\bbZ \times \Hom(\Weyl{G}{S}, k^\times)$.
    \end{proof}

   \section{The spectrum over an arbitrary commutative base}\label{sec:arb_base}

    In this section, we let $k$ denote an arbitrary commutative ring. We finally describe the Balmer spectrum of $\catK(G;k)$ using the result obtained in previous sections together with formal arguments related to Noetherian approximation.

    \begin{lemma}
        \label{lem-spc-as-a-lim}
        Let $k= \varinjlim_i k_i$ be a filtered colimit of commutative rings. Then
        \[
       \Spc( \catK(G;k)) \cong \varprojlim_i \Spc(  \catK(G;k_i))
        \]
    via the map induced by base change along $k_i\to k$.
    \end{lemma}

    \begin{proof}
        The filtered colimit of commutative rings induces a pseudo-natural transformation $(f_i\colon\on{K}_b(\perm(G;k_i))\to \on{K}_b(\perm(G;k_i))_i$ in the 2-category of essentially small tt-categories. We claim that $f_\sbull$  surjective on morphisms and detects isomorphism (see \cite[Definition 8.4]{Gal18}). Indeed, this is a consequence of the following observations. First, for any $H\leq G$, $k(G/H)\cong \varinjlim_i k_i(G/H)$ as $kG$-modules. Moreover, for any $H,K\leq G$ we have $\Hom_{\catK(G;k)}(k(G/H),k(G/K))\cong k(G/K)^H$. Then, for arbitrary $x,y\in \on{K}_b(\perm(G;k))$, $\Hom_{kG}(x,y)$ is bounded, and it is a finite sum of homomorphisms between permutation modules in each degree. This gives us the claim. Now, we invoke \cite[Proposition 8.5]{Gal18} to get
        \[
       \Spc( \on{K}_b(\mathsf{perm}(G;k)))) \cong \varprojlim_i \Spc( \on{K}_b(\mathsf{perm}(G;k_i)))
        \]
        Finally, recall that $\catK(G;k)\cong \on{K}_b(\perm(G;k))^\natural$, and idempotent completion does not change the spectrum. Hence the result follows.
    \end{proof}

    Recall the dg-algebra  $A\reg^\sbull(k)$ from \Cref{rec-tensor-algebra}.

   \begin{lemma}
        \label{lem-colim-tH}
        Let $k= \varinjlim_i k_i$ be a filtered colimit of commutative rings. Then
        \[
        A\reg^\sbull(k)\cong \varinjlim_i  A\reg^\sbull(k_i) \quad \mbox{and} \quad \tH(G;k)\cong \varinjlim_i \tH(G;k_i)
        \]
        as bigraded rings.
    \end{lemma}

    \begin{proof}
        Recall that $A\reg^\sbull(k)\cong A\reg^\sbull(\mathbb Z)\otimes k$. Hence the first isomorphism follows. Now, recall that $\tH(G;k)=\on{H}_\sbull(A\reg^\sbull(k))$. In this case, the homology commutes with filtered colimits, hence the second isomorphism is a consequence of the first one.
    \end{proof}

  \begin{proposition}
      \label{prop-colim-proj}
      Let $k=\varinjlim_i k_i$ be a filtered colimit of Noetherian rings. Then
      \[
      \Proj(\tH(G;k))\cong \varprojlim_i\Proj(\tH(G;k_i)).
      \]
  \end{proposition}

 \begin{proof}
     In view of \Cref{lem-colim-tH} and \cite[Lemma 2.12]{Lau23}, we have
     \[
     \Spec^h(\tH(G;k))\cong \varprojlim_i\Spec^h(\tH(G;k_i)).
     \]
     Hence, it suffices to verify that the transition maps
     \[
     f _{i,j}\colon \tH(G;k_i)\to \tH(G;k_j)
     \]
     satisfy that $\tH_+(G;k_j)$ is contained in the radical ideal generated by $f_{i,j}\tH_+(G;k_i)$.
     Now, by \Cref{prop:ideal_gen_by_forerunners}, we know that $\tH_+(G;k_i)$ is the radical ideal generated by the forerunners, and those are defined integrally, hence they go to forerunners under the transition maps. It follows that
     \[
    \sqrt{f_{i,j}(\tH_+(G;k_i))\tH(G;k_j)} = \sqrt{\tH_+(G;k_j)}
     \]
     as we wanted.
 \end{proof}

 \begin{theorem}
     \label{thm-spc=proj-general}
     Let $k$ be a commutative ring. Then
     \[
     \comp_G\colon \Spc(\catK(G;k))\to \Proj(\tH(G;k))
     \]
     is a homeomorphism. In particular, $x\reg$ is ample (see \Cref{def:ample}).
 \end{theorem}

 \begin{proof}
     By Noetherian approximation, we can write $k=\varinjlim_i k_i$ as a filtered colimit of commutative Noetherian rings. Now, the transition maps induce tt-functors by base change, and the naturality of the twisted comparison map we get the following commutative square.
     \[
     \begin{tikzcd}
         \Spc(\catK(G;k)) \ar[r] \ar[d]  &  \varprojlim_i \Spc(\catK(G;k)) \ar[d] \\
         \Proj(\tH(G;k)) \ar[r] & \varprojlim_i\Proj(\tH(G;k_i))
     \end{tikzcd}
     \]
     The top map is an homeomorphism by \Cref{lem-spc-as-a-lim}, the right vertical is an homeomorphism by \Cref{cor-spc=proj}, and the bottom map is a homeomorphism by \Cref{prop-colim-proj}. We conclude the result by the commutativity of the square. The final statement follows by definition.
 \end{proof}

    \begin{corollary}
        The collection $\{U(H)\}_{H \in \Sub_\Phi(G)/G}$ is an affine open cover of $\Spc(\catK(G;k))$. Moreover, if $k$ is (quasi-)local with unique maximal ideal $\frakm$, then the closed points of $\Spc(\catK(G;k))$ are separated by the open cover. Explicitly, one has that the closed point $\catM(\frakm; H)$ belongs to $U(K)$ if and only if $H,K \in \Sub_\Phi(G)$ are $G$-conjugate.
    \end{corollary}
    \begin{proof}
        By \Cref{lem-spc-as-a-lim}, the closed points of $\Spc(\catK(G;k))$ are precisely $\catM(\frakm; H)$ with $\frakm \in \Spec(k)$ maximal, and $\catM(\frakm; H) \in U(H)$, hence $\{U(H)\}$ forms an open cover. It is affine by \Cref{thm-spc=proj-general}. Conversely, \Cref{prop:forerunner_props} and \Cref{thm-spc=proj-general} together imply that if $K,H$ are non-conjugate, $\catM(\frakm;H)\not\in U(K)$, and the result follows.
    \end{proof}

    \begin{corollary}\label{coro-dirac-scheme}
        Let $\calO^\sbull_G$ denote the sheaf of $\bbZ$-graded rings given by sheafifying $U \mapsto \End^\sbull_{\catK(G;k)|_{U}}(\bbone)$. Then $(\Spc(\catK), \calO^\sbull_G)$ is a Dirac scheme in the sense of \cite{HP25}.
    \end{corollary}
    \begin{proof}
        The cover $\{U(H)\}_{H \in \Sub_\Phi(G)/G}$ is affine by \Cref{thm-spc=proj-general}.
    \end{proof}

    Perhaps it is surprising that only one well-chosen twist, $x\reg$, gives an ample family, rather than considering the entirety of the effective endotrivials as in \cite{BG25,Mil25c}. One can think of $x\reg$ packaging all the information of all the effective `irreducible' endotrivials (which are all tensor summands) into one package. The fact that the forerunners of this one endotrivial can distinguish between the $p$-local parts of the spectrum, combined with $p$-local cohomological surjectivity \Cref{prop:lifting_lemma}, and a fair amount of fiber functor yoga, give us all we need.

    \appendix

    \section{Finite generation of twisted cohomology rings in tensor-triangular geometry}\label{sec:twisted_coh_app}

    In this appendix, we give an abstract criteria for the finite generation of a general twisted cohomology ring for a tt-category. The strategy of the proof of the main result, \Cref{thm:general_th_thm}, and its supporting lemma \Cref{lem:inductive_lemma}, were found with the aid of generative AI in the context of permutation modules; see our AI disclosure for details.

    Let $k$ be a commutative ring and $\catK$ be an essentially small $k$-linear tt-category. We call a morphism $\bbone \to x[s]$ with $x \in \catK$ a \emph{global section} of $x$ of degree $s$.
    Throughout, we will assume $\catM \leq \Pic(\catK)$ is a finitely generated free submonoid of the Picard group of $\catK$ with (an ordered list of) free generators $x_1, \dots, x_m \in \Pic(\catK)$. Let $\calF = \{f_1, \dots, f_n\}$ be a list of global sections with target either (a shift of) $\bbone$ or $x_i$ for some $i$. We let $\calF_{x_i}\subseteq \calF$ denote the subset consisting of global sections with target $x_i$.

    \begin{definition}\label{def:most_general_twisted_coh}
        We let the \emph{$\catM$-graded twisted cohomology ring} of $\catK$ be the following multigraded $k$-algebra:
        \[\tH_\catM(\catK) = \tH_\catM(\catK;k) \coloneqq  \bigoplus_{s\in \bbZ} \bigoplus_{x \in \catM} \Hom_{\catK}(\bbone, x[s]).\]
        Then $\tH_\calM(\catK)$ is a graded-commutative $k$-algebra.
        More generally, for any $y \in \catK$, we define the following $\tH_\catM(\catK)$-module \[\tH_\catM(y) = \tH_\catM(y;k) \coloneqq \bigoplus_{s\in \bbZ} \bigoplus_{x \in \catM} \Hom_{\catK}(\bbone,  x\otimes y[s]).\]
        Given a finite collection of homogeneous elements $\calF = \{f_1, \dots, f_i \}\subset \tH_\catM(\catK)$, we write $R_\calF$ to denote the finitely generated graded subalgebra \[R_\calF \coloneqq  \langle f_1, \dots, f_i\rangle \subseteq \tH(\catK).\]
        Finally, we set $\on{H}^\sbull(\catK) = \on{H}^\sbull(\catK;k) \coloneqq  \End_{\catK}^\sbull(\bbone)$, and have an obvious inclusions and projections of graded $k$-algebras $\on{H}^\sbull(\catK) \hookrightarrow \tH_\catM(\catK) \twoheadrightarrow \on{H}^\sbull(\catK)$.
    \end{definition}

    \begin{remark}
        For twisted cohomology rings in greater generality, for instance when $\catM$ need not be free, then one runs into certain coherence issues preventing the associated multigraded ring from being graded-commutative. This issue is instead circumvented by instead constructing a \emph{graded commutative 2-ring} from the Picard 2-group of $\catK$. See \cite{DS14, Gal25} for a full treatment. However, because we assume our grading monoid is free abelian and finitely generated, we encounter no such issues.
    \end{remark}

    \begin{remark}\label{rmk:closed_immersion}
        The projection $\tH_\catM(\catK) \twoheadrightarrow \on{H}^\sbull(\catK)$ induces a closed immersion \[\Spec^h(\on{H}^\sbull(\catK)) \hookrightarrow \Spec^h(\tH_\catM(\catK))\] with open complement the support of the positively-graded part $\tH_{\catM^+}(\catK)$, where $\catM^+  = \catM \setminus \{e\}$ is the `irrelevant' monoid ideal.
    \end{remark}

    Our main result is the following criterion.

    \begin{theorem}\label{thm:general_th_thm}
        Assume $k$ is Noetherian. Let $\catM \leq \Pic(\catK)$ be freely generated by $x_1, \dots, x_m \in \Pic(\catK)$ and let $\calF = \{f_1, \dots, f_n\}\subseteq \tH_\catM(\catK)$ consist of homogeneous elements as above. Assume $\on{H}^\sbull(\catK)$ is a finitely generated $k$-algebra with a finite set of generators belonging to $\calF$, and for all $x \in \catK$, $\Hom_\catK^\sbull(\bbone, x)$ is a finitely generated $\on{H}^\sbull(\catK)$-module.

        Suppose for each $x_i \in \catM$, \[\calF_{x_i}\neq \varnothing, \text{    and   }\bigotimes_{f\in \calF_{x_i}} \cone(f) \cong 0 \in \catK.\]
        Then for every $y \in \catK$, $\tH_\catM(y)$ is finite over $R_\calF$. In particular, $\tH_\catM(\catK)$ is a finitely generated $k$-algebra, hence Noetherian.
    \end{theorem}

    \begin{recollection}\label{rec:comp_map_general}
        Since $\catM$ is assumed to be a finitely generated free monoid, by \cite[Remark A.22]{Gal25} $\tH_\catM(\catK)$ is a tightening of a certain graded-commutative 2-ring $\catR_\catM$, see \cite[Definition A.11]{Gal25} for details. Then by \cite[Corollary A.18]{Gal25}, we have a \emph{twisted comparison map} \[\comp_\catK\colon \Spc(\catK) \to \Spec^h(\tH_\catM(\catK)), \catP \mapsto \langle f \in \tH_\catM(\catK) \text{ homog.} \mid \cone(f) \not\in\catP \rangle.\]
    \end{recollection}

    The comparison map in generality was first constructed by Balmer \cite{Bal10}, and later extended to twisted cohomology (2-)rings by Dell'Ambrogio and Stevenson \cite{DS14, Gal25}, see also Sanders \cite{San13}.

    \begin{notation} \label{not:apdx_notation} We establish notation for the forthcoming proof.
        \begin{enumerate}
            \item Let $\catM\leq \Pic(\catK)$ be freely generated by $x_1, \dots, x_m$. We let $\widehat\catM \leq \catM$ denote the monoid ideal consisting of elements of the form $x_1^{a_1}\cdots x_m^{a_m}$ with $a_m \geq 1$, and let $\overline{\catM} =\langle x_1,\dots, x_{m-1}\rangle \leq \catM$. Note $\overline{\catM} \cong \catM/\widehat{\catM}.$ We have an obvious inclusion $\tH_{\overline{\catM}}(\catK) \subseteq \tH_{\catM}(\catK)$ which respects grading.
            \item Given a finite set of global sections $\calF = \{f_1, \dots, f_n\} \subset \tH_\catM(\catK)$, let $\overline{\calF}\subseteq \calF$ denote the subset of elements with target belonging up to shift to $\overline{\catM}$. Then $\on{R}_{\overline{\calF}} \subseteq \tH_{\overline{\catM}}(G)$ is a graded subalgebra.
            \item Given any graded $\tH_\catM(\catK)$-module \[M^{\sbull\sbull} = \bigoplus_{s\in \bbZ}\bigoplus_{x \in \catM}M^{s, x},\] we define a $\tH_\catM(\catK)$-submodule $\widehat{M}$ as follows:\[\widehat{M}^{\sbull\sbull} \coloneqq  \bigoplus_{s\in \bbZ}\bigoplus_{x \in \widehat\catM}M^{s, x}.\] We set $\overline{M}^{\sbull\sbull} = M^{\sbull\sbull}/\widehat{M}^{\sbull\sbull}$, then $\overline{M}^{\sbull\sbull}$ is a graded $\on{H}_{\overline{\catM}}^{\sbull\sbull}(\catK)$-module.
        \end{enumerate}
    \end{notation}

    We will rely on inductive arguments reminiscent of \cite[Lemma 12.12]{BG25} or \cite[Theorem 7.8]{Mil25c}. The following technical lemma is our main inductive tool.

    \begin{lemma}\label{lem:inductive_lemma}
        Assume $k$ is Noetherian. Let $\catM = \langle x_1, \dots, x_m\rangle \leq \Pic(\catK)$ be a finitely generated free abelian submonoid. Let $\calF = \{f_1, \dots, f_n\}$, and assume $f_n$ has target $x_m$ up to shift. Given $y \in \catK$, suppose:
        \begin{enumerate}
            \item $\tH_\catM(\cone(f_n) \otimes y)$ is a finitely generated $R_\calF$-module;
            \item $\tH_{\overline{\catM}}(y)$ is a finitely generated $R_{\overline{\calF}}$-module.
        \end{enumerate}
        Then $\tH_\catM(y)$ is a finitely-generated $R_\calF$-module.
    \end{lemma}
    \begin{proof}
        We will use assumptions (a) and (b) to show that a quotient of $\tH_\catM(y)$ is a finitely generated $R_\calF$-module, then give a Nakayama-type argument to lift finite generation to $\tH_\catM(y)$. Suppose $f_n$ has target $x_m[s]$ with $s \in \bbZ$. Then we have an exact triangle \[\bbone \xrightarrow{f_n} x_m[s] \to \cone(f_n) \to \bbone[1],\] and tensoring with $y$ produces the triangle \[y \xrightarrow{f_n \otimes \id_y} x_m \otimes y[s] \to \cone(f_n) \otimes y \to y[1].\] Observe that $\tH_\catM(-)$ is a covariant functor $\catK\to \mathsf{Mod}(\tH_\catM(\catK))$.
        Applying it to the above triangle yields the following exact sequence of $\tH_\catM(\catK)$-modules \[\tH_\catM(y) \xrightarrow{\cdot f_n} \tH_\catM(y) \to \tH_\catM(\cone(f_n) \otimes y),\] where multiplication by $f$ shifts $\bbZ$-degree by $s$ and $\catM$-degree by $x_m$. Consider the quotient $M^{\sbull\sbull}$ arising from $\cdot f_n$, that is, \[M^{\sbull\sbull} \coloneqq  \tH_\catM(y)/ \im(\cdot f_n).\]
        The $\tH_\catM(\catK)$-submodule $\widehat{M}^{\sbull\sbull}$ injects (with a grading shift) into $\on{H}^{\sbull\sbull}_\catM(\cone(f_n) \otimes y)$ by the above exact triangle. Since $\on{H}^{\sbull\sbull}_\catM(\cone(f_n) \otimes y)$ is a finitely generated $R_\calF$-module by hypothesis, and $R_\calF$ is finitely generated, the submodule $\widehat{M}^{\sbull\sbull}$ is also a finitely generated $R_\calF$-module.

        Now, since the map $\cdot f_n$ raises the $\catM$-degree by $x_m$, $\im(\cdot f_n)$ misses the degrees of $\tH_\catM(y)$ associated to $\overline{\catM}\leq\catM$. Therefore, essentially by construction we have an isomorphism of $\tH_\catM(\catK)$-modules\[M^{\sbull\sbull}/\widehat{M}^{\sbull\sbull} \cong \tH_{\overline{\catM}}(y),\] where the right-hand side is a $\tH_\catM(\catK)$-module via the projection $\tH_\catM(\catK) \twoheadrightarrow \tH_{\overline{\catM}}(\catK).$ By hypothesis, $\tH_{\overline{\catM}}(y)$ is a finitely generated $R_{\overline{\calF}}$-module, hence a finitely generated $R_\calF$-module, and since $\widehat{M}^{\sbull\sbull}$ is also a finitely generated $R_\calF$-module, $M^{\sbull\sbull}$ is a finitely generated $R_\calF$-module.

        It remains to show that we may lift finite generation of $M^{\sbull\sbull} = \tH_\catM(y)/\im(\cdot f_n)$ to $\tH_\catM(y)$. Choose finitely many homogeneous generators of $a_1, \dots a_l$ of $M^{\sbull\sbull}$ over $R_\calF$ and corresponding homogeneous lifts $b_1,\dots, b_l$ of same degree in $\tH_\catM(y)$. Let \[N^{\sbull\sbull} = \langle b_1, \dots, b_l\rangle_{R_\calF} \subseteq \tH_\catM(y),\] then we have $\tH_\catM(y) = N^{\sbull\sbull} +  f_n\cdot \tH_\catM(y).$ Set $P^{\sbull\sbull} \coloneqq  \tH_\catM(y)/N^{\sbull\sbull}.$ It suffices to show $P^{\sbull\sbull} = 0$. From the previous equality, we have $P^{\sbull\sbull} =  f_n\cdot P^{\sbull\sbull}$.

        Since $\im(\cdot f_n) \subseteq \tH_{\widehat{\catM}}(y)$, we have equality $\on{H}_\catM^{s,x}(y) = N^{s,x}$ for all $s \in \bbZ$ and $x \in \overline{\catM} \leq \catM$, and consequently, $P^{s,x} = 0$ for all $x \in \overline{\catM} \leq \catM$. But since $P^{\sbull\sbull} =  f_n\cdot P^{\sbull\sbull}$, this implies $P^{s,x} = 0$ for all $x \in \catM$ with one $x_m$-constituent, since multiplication by $f_n$ increases $x_m$-degree by one. Repeating this argument ad nauseam gives $P^{\sbull\sbull} = 0$, as desired.
    \end{proof}

    We can now prove the finite generation criterion, inducting with \Cref{lem:inductive_lemma}.

    \begin{proof}[Proof of \Cref{thm:general_th_thm}]
        We induct on the number of generators of $\catM = \langle x_1, \dots, x_m\rangle$. The case of $m = 0$ follows immediately by hypothesis. Assume $m \geq 1$ and that the result holds for $\overline{\catM} \leq \catM$. Then $\tH_{\overline{\catM}}(y)$ is a finitely generated $R_{\overline{\calF}}$-module for every $y \in \catK$. Let $\{f_1, \dots f_n\} = \calF_{x_m} \subset \calF$.

        We recursively define the following family of objects of $\catK$: set $y_0 \coloneqq  y$ and for $i \geq 1$, we set $y_i \coloneqq  \cone(f_i)\otimes y_{i-1}$. By hypothesis, $y_n = 0$, so $\tH(y_n) = 0$, hence trivially a finitely generated $R_\calF$-module. The induction hypothesis implies $\tH_{\overline{\catM}}(y_{n-1})$ is a finitely generated $R_{\overline{\calF}}$-module, so we may apply \Cref{lem:inductive_lemma} to obtain that $\tH_{\catM}(y_{n-1})$ is a finitely generated $R_\calF$-module. Recursively applying \Cref{lem:inductive_lemma} shows $\tH_{\catM}(y)$ is a finitely generated $R_\calF$-module, as desired. The last statement is standard commutative algebra.
    \end{proof}

    \bibliography{bib}
    \bibliographystyle{alpha}

\end{document}